\documentclass[12pt]{amsart}
\usepackage{bm,amsmath,mathrsfs,graphicx,amscd}
\begin{document}
\newtheorem{theorem}{Theorem}[section]
\newtheorem{lemma}[theorem]{Lemma}
\newtheorem{definition}[theorem]{Definition}
\newtheorem{corollary}[theorem]{Corollary}
\newtheorem{remark}[theorem]{Remark}
\numberwithin{equation}{section}
\numberwithin{table}{section}
\numberwithin{figure}{section}
\allowdisplaybreaks[4]
\def\sss{\scriptscriptstyle}
\def\R{\mathbb{R}}
\def\O{\Omega}
\def\LT{{L_2(\O)}}
\def\HOne{{H^1(\O)}}
\def\SS{H^2(\O)\cap H^1_0(\O)}
\def\cT{\mathcal{T}}
\def\cV{\mathcal{V}}
\def\cE{\mathcal{E}}
\def\cF{\mathcal{F}}
\def\hF{h_{\scriptscriptstyle F}}
\def\jumpF#1{[\hspace{-2pt}[#1]\hspace{-2pt}]}
\def\jumpE#1{[\hspace{-2pt}[#1]\hspace{-2pt}]}
\def\jump#1{[\hspace{-2pt}[#1]\hspace{-2pt}]}
\def\p{\partial}
\def\Cordes{1-(1-\epsilon)^\frac12}
\def\DeltaF{\Delta_{\scriptscriptstyle F}}
\def\nablaF{\nabla_{\hspace{-2pt}\scriptscriptstyle F}}
\def\vI{v_{\sss I}}
\def\vT{v_{\sss T}}
\def\tE{\tilde{E}}
\def\G{\Gamma}
\def\sC{\mathscr{C}}
\def\sQ{\mathscr{Q}}
\def\sF{\mathscr{F}}
\def\LHat{L_{\hat T}}
\def\TSpace{(H^2\times H^1)(\p T)}
\def\hatTSpace{(H^2\times H^1)(\p \hat{T})}
\def\FSoace{(H^2\times H^1)(\p F)}
\def\fT{f_{v,{\sss T}}}
\def\fpT{f_{v,\sss \p T}}
\def\fF{f_{v,\sss F}}
\def\fpF{f_{v,\sss \p F}}
\def\LTT{{L_2(\hat{T})}}
\def\hT{h_{\sss T}}
\def\gT{g_{v,{\sss T}}}
\def\qF{q_{v,\sss \p F}}
\def\gF{g_{v,\sss F}}
\def\bw{\bm{w}}
\def\argmin{\mathop{\rm argmin}}
\def\TPoly{(H^2\times H^1)_{k,k-1}(\p T)}
\def\tTPoly{(H^2\times H^1)_{k,k-1}(\p \tilde T)}
\def\FPoly{(H^2\times H^1)_{k,k-1}(\p F)}
\def\hatTPoly{(H^2\times H^1)_{k,k-1}(\p\hat T)}
\def\cQ{\mathcal{Q}}
\def\cB{\mathcal{B}}
\def\ET{H^\frac52(T)}
\def\Tr{\mathrm{Tr}\,}
\def\HTO{H^2(\p T)\times H^1(\p T)}
\def\LTBT{L_2(\p T)}
\def\bn{\bm{n}}
\def\bz{\bm{z}}
\def\VE{\mathscr{V}^k(T)}
\def\hatVE{\mathscr{V}^k(\hat{T})}
\def\VEF#1{\cB^k(F_{#1})}
\def\RPTwo{\R_+\times\R_+}
\def\RPThree{\R_+\times\R_+\times\R_+}
\def\vF{v_{\sss F}}
\def\sG{\mathscr{G}}
\def\tT{\tilde T}
\def\tTSpace{(H^2\times H^1)(\p \tT)}
\def\tp{\tilde p}
\def\tf{\tilde f}
\def\tg{\tilde g}
\def\tn{\tilde n}
\def\te{\tilde e}
\def\bt{\bm{t}}
\def\tF{\tilde{F}}
\def\COne{E^1_0(\p T)}
\def\mfg{\mathfrak{g}}
\def\GT{G_{\sss T, p_\dag}}
\def\HT{H_{\sss T, p_\dag}}
\def\MT{M_{\sss T,p_\dag}}
\def\sp{\hspace{1pt}}
  \title{Virtual Enriching Operators}
\author{Susanne C. Brenner}
\address{Susanne C. Brenner, Department of Mathematics and Center for
Computation and Technology, Louisiana State University, Baton Rouge}
\email{brenner@math.lsu.edu}
\author{Li-yeng Sung}
 \address{Li-yeng Sung,
 Department of Mathematics and Center for Computation and Technology,
 Louisiana State University, Baton Rouge, LA 70803, USA}
\email{sung@math.lsu.edu}
\thanks{This work was supported in part
 by the National Science Foundation under Grant No.
 DMS-16-20273.}
\begin{abstract}
 We construct bounded linear operators that map $H^1$ conforming Lagrange finite element spaces
 to $H^2$ conforming virtual element spaces in two and three dimensions.
 These operators are useful for the analysis of nonstandard finite element methods.
\end{abstract}
\maketitle
\section{Introduction}\label{sec:Introduction}
 Let $\O\in\R^d$ ($d=2,3$) be a bounded polygonal/polyhedral domain, $\cT_h$ be a simplicial
 triangulation of $\O$ and $V_h\subset H^1(\O)$ be the Lagrange $P_k$ finite element space
 with $k\geq 3$. The mesh dependent semi-norm $\|\cdot\|_h$ is defined by
\begin{equation}\label{eq:hNorm}
  \|v\|_h^2=\|D_h^2v\|_\LT^2+J(v,v) \qquad\forall\,v\in V_h,
\end{equation}
 where $D_h^2v$ is the piecewise Hessian of $v$ with respect to $\cT_h$, and
\begin{alignat}{3}
  J(w,v)&=\sum_{e\in\cE_h^i}h_e^{-1}\int_e \jump{\p w/\p n}\jump{\p v/\p n}ds
  &\qquad&\text{for $d=2$},  \label{eq:2DJump}\\
  J(w,v)&=\sum_{F\in\cF_h^i}\hF^{-1}\int_F \jump{\p w/\p n}\jump{\p v/\p n}dS
  &\qquad&\text{for $d=3$}.\label{eq:3DJump}
\end{alignat}
 Here $\cE_h^i$ (resp., $\cF_h^i$)  is the set of interior edges (resp., faces) of $\cT_h$,
 $h_e$ (resp., $\hF$) is the diameter of the edge $e$ (resp., face $F$),
 $\jump{\p v/\p n}$ is the jump of the normal derivative across an edge $e$ (resp., face $F$),
 and $ds$ (resp., $dS$) is the infinitesimal length (resp., area).
\par
 Our goal is to construct a linear operator $E_h:V_h\longrightarrow H^2(\O)$ such that
\begin{alignat}{3}
 \|v-E_hv\|_h&\leq C_\sharp \sqrt{J(v,v)}&\qquad&\forall\,v\in V_h, \label{eq:EhError}\\
 \sum_{\ell=0}^2h^\ell |\zeta-E_h\Pi_h\zeta|_{H^\ell(\O)}&\leq C_\flat h^{k+1}
  |\zeta|_{H^{k+1}(\O)}
 &\qquad&\forall\,\zeta\in H^{k+1}(\O),\label{eq:EhPihError}
 \end{alignat}
 where $\Pi_h:C(\bar\O)\longrightarrow V_h$ is the Lagrange nodal interpolation operator, and
 the positive constants $C_\sharp$ and $C_\flat$ only depend on the shape regularity of
 $\cT_h$ and $k$. Moreover, the operator $E_h$ maps $V_h\cap H^1_0(\O)$ into $\SS$.
\par
 Enriching operators that satisfy \eqref{eq:EhError} and \eqref{eq:EhPihError} are useful
 for {\em a priori} and {\em a posteriori} error analyses for fourth order elliptic problems
 \cite{BGS:2010:C0IP,GHV:2011:DG,Brenner:2012:C0IP,BGSZ:2017:AdaptiveVI4,BSung:2017:State},
 and they also play an important role in fast solvers for fourth order problems
 \cite{Brenner:1996:NCSchwarz,Brenner:1999:CNM,BW:2005:TLDG4}.  A recent application to
 Hamilton-Jacobi-Bellman equations can be found in \cite{NW:2017:MT}.
\par
 In the two dimensional case, one can construct $E_h$ through the $C^1$ macro finite elements
 in \cite{CT:1965:Element,Ciarlet:1974:HCT,Percell:1976:CT,Peano:1976:Macro,DDPS:1979:Macro}.
 This was carried out in \cite{BGS:2010:C0IP} for the quadratic element and in \cite{GHV:2011:DG}
 for higher order elements.  However macro elements of order higher than 3 are not available in
 three dimensions and therefore this approach can only be carried out for quadratic and cubic
 Lagrange elements (cf. \cite{NW:2017:MT})
  using the three dimensional cubic Clough-Tocher macro element from \cite{WF:1987:CT}.
\par
 We take a different approach in this paper by connecting the $k$-th order
 Lagrange finite element space to
 a $k$-th order $H^2$ conforming virtual element space.
 In two dimensions such spaces are already in the literature
 \cite{BM:2013:VEM,CM:2016:VEM},
 and we will develop
 a version of three dimensional $H^2$ conforming virtual element spaces that are sufficient
 for the construction of $E_h$.
%
\begin{remark}\label{rem:Quadratic}\rm
  The assumption that the order of the Lagrange finite element space is at least 3 allows a
  uniform construction of $E_h$.  For the
  quadratic Lagrange finite element space we can simply
  take $E_h$  to be the restriction of the cubic enriching operator.
\end{remark}
\par
 The rest of the paper is organized as follows.  The construction of $E_h$ in two dimensions
 is carried out in Section~\ref{sec:2D}, followed by the construction in three dimensions in
  Section~\ref{sec:3D} and some concluding remarks in Section~\ref{sec:Conclusions}.
    Appendix~\ref{append:ITT}
   contains some technical results concerning inverse trace
  theorems that are needed for the construction of $H^2$ conforming virtual elements.
\par
 A list of notations and conventions that will be used throughout the paper is provided here for convenience.
\begin{itemize}
\item A  polygon/polyhedron is an open subset in $\R^2/\R^3$,
 an edge of a polygon/polyhedron does not include the endpoints and a face of a polyhedron does not
 include the vertices and edges.  These conventions apply in particular to triangles and tetrahedra.
\item Let $G$ be an open line segment, a triangle or a tetrahedron, and $k$ be an integer.
 $P_k(G)$ is the space of polynomials of total degree $\leq k$ restricted to $G$ if $k\geq0$ and
 $P_k(G)=\{0\}$ if $k<0$.  We say that two functions $u$ and $v$ defined on $G$ have identical moments up
 to order $\ell$ if the integral of $(u-v)q$ on $G$ vanishes for all $q\in P_\ell(G)$.  The orthogonal projection
 from $L_2(G)$ onto $P_k(G)$ is denoted by $Q_{G,k}$.
\item $\cV_h$ is the set of all the vertices of the triangles/tetrahedra in $\cT_h$, $\cV_h^i$ is the
  set of vertices in $\O$ and $\cV_h^b$ is the set of vertices on $\p\O$.
\item $\cE_h$ is the set of all the edges of the triangles/tetrahedra in $\cT_h$,
 $\cE_h^i$ is the set of edges in $\O$ and $\cE_h^b$ is the set of edges that are subsets of $\p\O$.
\item $\cF_h$ is the set of all the faces of the tetrahedra in $\cT_h$,
 $\cF_h^i$ is the set of faces in $\O$ and $\cF_h^b$ is the set of faces that are subsets of $\p\O$.
\item $\cT_p$ is the set of all the triangles/tetrahedra in $\cT_h$ that share $p$ as a common vertex.
\item $\cT_e$ is the set of all the triangles/tetrahedra in $\cT_h$ that share $e$ as a common edge.
\item $\cT_F$ is the set of all the tetrahedra in $\cT_h$ that share $F$ as a common face.
%
%
\item $\cF_e$ is the set of the faces of the tetrahedra in $\cT_h$ that share $e$ as a common edge.
\item If $v$ is a function defined on a triangle/tetrahedron, then $v_e$ (resp., $\vF$) is the
 restriction of $v$ to an edge $e$ (reps., a face $F$).
%
\item If $v$ is a function defined on a triangle/tetrahedron, then $\p v/\p n$ denotes the
 outward normal derivative of $v$ along $\p T$.  In the case of a triangle (resp. tetrahedron),
 $\p v/\p n$ is double-valued at the vertices (resp., edges) of $T$.
\item If $e$ is an edge of the triangle $T$, then $\bn_{e,\sss T}$ is the unit vector
  orthogonal to $e$ and pointing towards the outside of $T$.  If $e$ is an edge of a face $F$ of a
  tetrahedron $T$, then $\bn_{e,\sss F}$ is the unit vector tangential to $F$, orthogonal to $e$
  and pointing towards the outside of $F$.
\item If $F$ is a face of the tetrahedron $T$, then $\bn_{\sss F,\sss T}$ is the unit vector
 orthogonal to $F$ and pointing towards the outside of $T$.
%
\end{itemize}

\section{The Two Dimensional Case}\label{sec:2D}
 The construction of $E_h$ is based on the characterizations of trace spaces associated with a
 triangle and
 the construction of polynomial data on the skeleton of $\cT_h$
 that satisfy these characterizations on all the triangles of $\cT_h$.
\subsection{Trace Spaces for a Triangle}\label{subsec:2DLocalSpaces}
  Let $T$ be a triangle with vertices
 $p_1$, $p_2$ and $p_3$, $e_i$ be the edge of $T$ opposite $p_i$, $\bn_i$ be the unit outer normal
 along $e_i$, and $\bm{t}_i$ be the counterclockwise unit tangent of $e_i$.
%
%
  Let $\ell$ be a nonnegative number. A function $u$ belongs to
 the piecewise Sobolev space $H^\ell(\p T)$
  if and only if $u_i$, the restriction of $u$ to $e_i$, belongs to $H^\ell(e_i)$ for $1\leq i\leq 3$.
  It follows from the Sobolev Embedding Theorem \cite[Theorem~4.12]{ADAMS:2003:Sobolev}
  that we can define
  a linear operator $\Tr:H^2(T)\longrightarrow H^\frac32(\p T)\times H^\frac12(\p T)$  by
\begin{equation}\label{eq:Trace}
 \Tr \zeta=(\zeta,\p\zeta/\p n)\big|_{\p T},
\end{equation}
 where the restrictions of $\zeta$ and $\p \zeta/\p n$ are in the sense of trace and defined
 piecewise with respect to the edges.  For the subspace $\ET$ of $H^2(T)$,
   we have $\Tr\ET\subset H^2(\p T)\times H^1(\p T)$.  Our first task is to identify the image of $\ET$.
\begin{definition}\label{def:2DLocalTraceSpace}\rm
  A pair $(f,g)\in \HTO$ belongs to the space $\TSpace$ if and only if the following conditions are satisfied:
\begin{alignat}{3}
  f_{j}(p_i)&=f_{k}(p_i)&\qquad&
   \text{for}\;1\leq i\leq 3\;\text{and}\;j,k\in \{1,2,3\}\setminus\{i\},\label{eq:2DC1}\\
\intertext{and there exist $\bw_1,\bw_2, \bw_3\in \R^2$
 $($which depend on $(f,g))$ such that}
  (\p f_j/\p t_j) (p_i)&=\bw_i\cdot  {\bm t}_j &\qquad&\text{for}\;1\leq i\leq 3\;
  \text{and}\;j\in\{1,2,3\}\setminus\{i\},
  \label{eq:2DC2}\\
   g_j(p_i)&=\bw_i\cdot{\bm n}_j &\qquad&\text{for}\;1\leq i\leq 3\;\text{and}\;
    j\in\{1,2,3\}\setminus\{i\}.\label{eq:2DC3}
\end{alignat}
\end{definition}
%
\par
 Note that the compatibility conditions \eqref{eq:2DC2}--\eqref{eq:2DC3} are equivalent to
\begin{equation}\label{eq:2DCompatibleGradient}
   (\p f_j/\p t_j)\bt_j+ g_j\bn_j=(\p f_k/\p t_k)\bt_k+ g_k\bn_k \qquad\text{at} \quad p_i
\end{equation}
 {for} $1\leq i\leq 3$ and $j,k\in\{1,2,3\}\setminus\{i\}$.
\par
 It follows from the Sobolev Embedding Theorem that $\Tr\zeta\in\TSpace$ for $\zeta\in\ET$, where
 $\bw_i=\nabla\zeta(p_i)$, and we can recover $\nabla\zeta$ on $\p T$
 from $(f,g)=\Tr\zeta$ by
\begin{equation}\label{eq:2DGradient}
 \nabla\zeta=(\p f_i/\p t_i)\bt_i+ g_i\bn_i \qquad\text{on $e_i$ for $1\leq i\leq 3$.}
\end{equation}
 \par
  We want to show that in fact $\Tr\ET=\TSpace$.  For this purpose
  it is useful to construct a linear isomorphism $\Phi^*:\TSpace\longrightarrow\tTSpace$ such that
 \begin{equation}\label{eq:CD}
   \Tr (\zeta\circ\Phi)=\Phi^*(\Tr\zeta) \qquad\forall\,\zeta\in\ET,
 \end{equation}
 where $\Phi$ is an orientation preserving affine transformation that maps the triangle $\tT$ onto $T$.
 We assume that $\Phi$ maps the vertex $\tp_i$ of $\tT$ to the vertex $p_i$ of $T$ and hence
 it also maps the edge $\te_i$ of $\tT$ to the edge $e_i$ of $T$.
 \par
 First we note that, by the chain rule,
\begin{equation}\label{eq:ChainRule}
  \nabla (\zeta\circ\Phi)=J_\Phi^t(\nabla\zeta\circ\Phi) \qquad\forall\,\zeta\in\ET,
\end{equation}
 where $J_\Phi$ (a constant $2\times 2$ matrix with a positive determinant) is the Jacobian of $\Phi$ .
\par
 Let $(f,g)\in\ET$.
 Motivated by \eqref{eq:2DGradient}--\eqref{eq:ChainRule}, we define
$ \Phi^*(f,g)=(\tf,\tg)$,
 where
\begin{equation}\label{eq:PullBack1}
   \tf=f\circ\Phi,
 \end{equation}
  and
\begin{equation}\label{eq:PullBack2}
 \tg=J_\Phi^t(\mfg\circ\Phi)\cdot\tilde\bn_i \qquad\text{on $\;\te_i\;$ for $1\leq i\leq 3$},
\end{equation}
 where $\tilde\bn_i$ is the outward pointing unit normal along the edge $\te_i$ and
 the vector field $\mfg$ on $\p T$ is given by
\begin{equation}\label{eq:ArtifcialGradient}
 \mfg=(\p f_i/\p t_i)\bt_i+ g_i\bn_i \qquad\text{on $\;e_i\;$ for $1\leq i\leq 3$.}
\end{equation}
\par
 It is straightforward to check that $(\tf,\tg)\in\tTSpace$, $\Phi^*:\TSpace\longrightarrow\tTSpace$ is a bijection,
  and that
 \eqref{eq:CD} follows from \eqref{eq:2DGradient} and
 \eqref{eq:ChainRule}--\eqref{eq:ArtifcialGradient}.
 \par
  We are now ready to characterize $\Tr \ET$.
\begin{lemma}\label{lem:2DChracterization}
  The image of $\ET$ under $\Tr$ is the space $\TSpace$.
\end{lemma}
\begin{proof} We already know that $\Tr\ET\subset\TSpace$.  In the other direction,
 we want to construct
  $\zeta\in\ET$ that satisfies \eqref{eq:Trace} for a
 given $(f,g)\in\TSpace$.
\par
   If $f$ and $g$ vanish near the vertices, we can use the
 operator $L_1$ in Lemma~\ref{lem:2DLifting} and cut-off functions to obtain $\zeta$.
 Therefore,  by using a partition of unity, we can reduce the construction to a neighborhood of a vertex
 and, by an affine transformation (cf. \eqref{eq:CD}), we can further assume that
 the angle at the vertex is a right angle.
 The existence of $\zeta$ near such a vertex then follows from Lemma~\ref{lem:2DLocalTrace}.
\end{proof}
%
\subsection{Affine Invariant $\bm{H^2}$ Virtual Element Spaces}\label{subsec:2DVEM}
 The construction of the virtual element spaces involves  polynomial subspaces of $\TSpace$.
\begin{definition}\label{def:BdryPolySpace}\rm Let $T$ be a triangle.
  We will denote the intersection of
   $\TSpace$ and $P_{k}(\p T)\times P_{k-1}(\p T)$ by $\TPoly$.
\end{definition}
\begin{remark}\label{rem:2DTPolyDimension}\rm
  It follows from the compatibility conditions \eqref{eq:2DC1}--\eqref{eq:2DC3} that
  $(f,g)\in \TPoly$ is determined by (i) the values of $f$ at the vertices, (ii)
  the tangential derivatives of $f$ at the vertices, (iii) moments of $f$ on $e_i$ up to order
  $k-4$ that together with (i) and (ii) determine $f_i\in P_k(e_i)$, and (iv) moments of
  $g$ up to order $k-3$ that together with (ii)  (through
  \eqref{eq:2DC3}) determine $g_i\in P_{k-1}(e_i)$.  These degrees of freedom (dofs) are depicted in
  Figure~\ref{fig:2DTSpace} for $k=3$ and $4$, where (i) the values of $f$ at the vertices and
  the moments of $f$ on the edges are
  represented by solid dots, and (ii) the tangential derivatives of $f$ at the vertices and
  the moments of $g$ on the edges are represented by arrows.
  Altogether we have
%
 $ \mathrm{dim} \TPoly=6(k-1)$.
\end{remark}
\begin{figure}[h]
\includegraphics[width=3in]{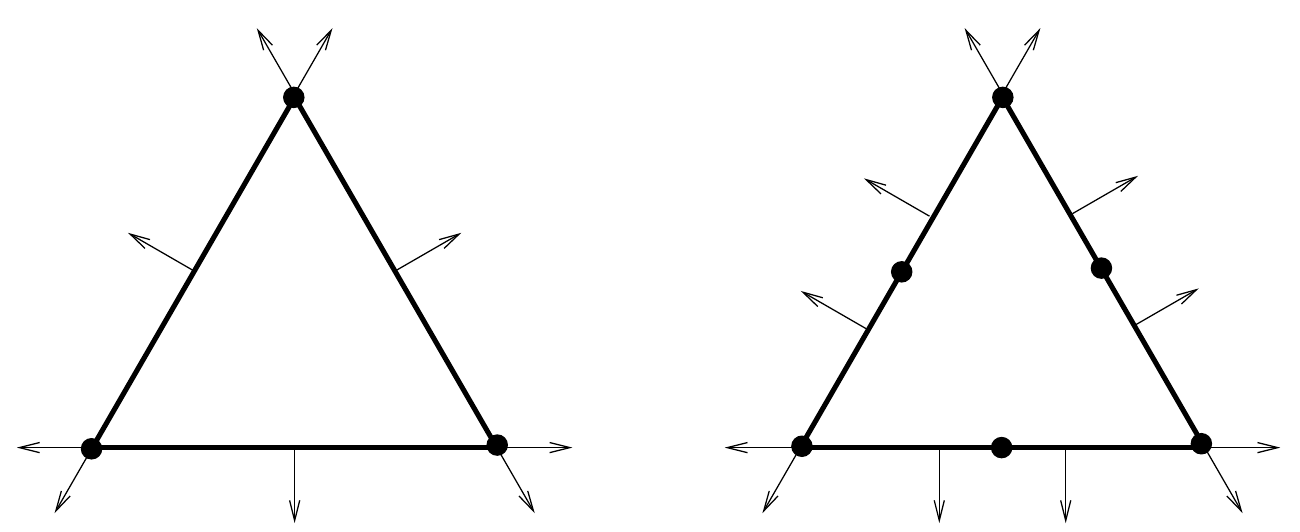}
\caption{Degrees of freedom for $(H^2\times H^1)_{3,2}(\p T)$ and $(H^2\times H^1)_{4,3}(\p T)$}
\label{fig:2DTSpace}
\end{figure}
\begin{remark}\label{rem:2DPolyData}\rm
  Since polynomial spaces are preserved by an affine transformation,
  the map $\Phi^*:\TSpace\longrightarrow \tTSpace$
  defined by \eqref{eq:PullBack1}--\eqref{eq:PullBack2}
  maps $\TPoly$ one-to-one and onto $\tTPoly$.
\end{remark}
\subsubsection{Virtual Element Spaces on the Reference Triangle}\label{subsubsec:2DVEMReference}
 We begin with  a simple well-posedness result for the biharmonic problem.
\begin{lemma}\label{lem:LocalWellPosedness}
 Given any $(f,g)\in \TSpace$ and $\rho\in L_2(T)$, there
 exists a unique $\xi\in H^2(T)$ such that
 \begin{equation}\label{eq:WeakBiharmonic}
  (\Delta\xi,\Delta z)_{L_2(T)}=(\rho,z)_{L_2(T)}\qquad\forall\,z\in H^2_0(T)
  \quad\text{and}\quad  \Tr \xi= (f,g).
 \end{equation}
\end{lemma}
\begin{proof}
  Let $\zeta\in \ET$ satisfy \eqref{eq:Trace} and $\eta\in H^2_0(T)$ be defined by
\begin{equation*}
  (\Delta\eta,\Delta z)_{L_2(T)}=(\rho-\Delta \zeta,z)_{L_2(T)}\qquad\forall\,z\in H^2_0(T).
\end{equation*}
 Then $\xi=\eta+\zeta$ is the unique solution of (2.6).
\end{proof}
 Let $\hat{T}$ be the reference triangle with vertices $(0,0)$, $(1,0)$ and $(0,1)$.
 In view of Lemma~\ref{lem:LocalWellPosedness} and the fact that $\hatTPoly$ is a subspace of $\hatTSpace$, we can
 now define the reference virtual element spaces $\hatVE$,
  which are identical to the virtual element spaces in \cite{BM:2013:VEM}
  for the special case of the reference triangle.
\begin{definition}\label{def:2DReferenceVEM}\rm
  A function $\hat\xi\in H^2(\hat{T})$ belongs to the
  virtual element space $\hatVE$
  if and only if
 $\Tr\hat\xi\in\hatTPoly$ and the distributional derivative
 $\Delta^2 \hat\xi$ belongs to $P_{k-4}(\hat{T})$, i.e., there exists $\hat\rho\in P_{k-4}(\hat{T})$
 such that
\begin{equation}\label{eq:RefVE}
 (\Delta\hat\xi,\Delta \hat z)_\LTT=(\hat\rho,\hat z)_\LTT \qquad\forall\,\hat z\in H^2_0(\hat{T}).
\end{equation}
\end{definition}
\begin{remark}\label{rem:VEMDimension}\rm
  According to
  Remark~\ref{rem:2DTPolyDimension} and Lemma~\ref{lem:LocalWellPosedness},
  we have
\begin{equation}\label{eq:2DVEMDimension}
  \mathrm{dim}\,\hatVE=\mathrm{dim}\,\hatTPoly
     +\mathrm{dim}\,P_{k-4}(\hat{T})=\frac{k^2+7k-6}{2}.
\end{equation}
\end{remark}
\par
 The following result is well-known
 (cf. \cite{BM:2013:VEM}).  We provide a proof here for self-containedness.
\begin{lemma}\label{lem:BM}
  A function $\hat\xi$ in $\hatVE$ is uniquely determined by
  $\Tr\hat\xi\in\hatTPoly$ and $Q_{\hat{T},k-4}\hat\xi\in P_{k-4}(\hat{T})$.
\end{lemma}
\begin{proof} In view of Remark~\ref{rem:VEMDimension}, it suffices to show that $\hat\xi=0$ is
 the only function in $\hatVE$ with the properties that
  $\Tr\hat\xi=(0,0)$ and $Q_{\hat{T},k-4}\hat\xi=0$.  Indeed, using
  integration by parts and the fact that the distributional derivative
  $\Delta^2\hat\xi\in P_{k-4}(\hat{T})$, we have
\begin{equation*}
  (\Delta \hat\xi,\Delta \hat\xi)_\LTT=(\Delta^2\hat\xi,\hat\xi)_\LTT=0.
\end{equation*}
 Therefore $\hat\xi\in H^2(\hat{T})$ is a harmonic function that vanishes on
 $\p \hat{T}$ and hence $\hat\xi=0$.
\end{proof}
\begin{remark}\label{rem:ShortCut}\rm
 The definition of the virtual element space $\hatVE$ relies on the fact that
 $\hatTPoly$ is a subspace of $\Tr H^\frac52(\hat{T})\subset \Tr H^2(\hat{T})$.
  One can show by using macro elements of order $k$
  that a pair $(\hat{f},\hat{g})\in P_k(\p \hat{T})\times P_{k-1}(\p \hat{T})$ satisfying the compatibility
  conditions \eqref{eq:2DC1}--\eqref{eq:2DC3} automatically belongs to $\Tr H^2(\hat{T})$.  Hence
  Lemma~\ref{lem:2DChracterization} is not necessary for the definition of the virtual element
  space $\hatVE$ in two dimensions.  However, the definition of the virtual element spaces in three dimensions
  requires the characterization of the trace of $H^\frac52(\hat{T})$ for the reference tetrahedron $\hat{T}$, since
  macro elements of arbitrary order are not available.  The approach
  here provides a preview of the three dimensional case.
\end{remark}

\subsubsection{Virtual Element Spaces for a General Triangle}\label{subsubsec:2DVEMGeneral}
 We now define $\VE$ for an arbitrary triangle $T$ in terms of $\hatVE$.
\begin{definition}\label{def:2DVEM}\rm
 Let $T$ be an arbitrary triangle and
 $\Phi$ be an orientation preserving affine transformation that maps $\hat{T}$ onto $T$.
 Then $\xi\in \VE$ if and only if $\xi\circ\Phi\in\hatVE$.
\end{definition}
\begin{remark}\label{rem:PolynomialSubspace}\rm The definition of $\VE$ is independent of the choice of $\Phi$.
  The polynomial space $P_{k}(T)$ is a subspace of $\VE$ since $P_k(\hat T)$ is obviously a subspace
  of $\hatVE$.  The dimension of $\VE$ is also given by the formula in \eqref{eq:2DVEMDimension}.
\end{remark}
\par
  We have an analog of Lemma~\ref{lem:BM}.
\begin{lemma}\label{lem:GeneralBM}
   A function $\xi$ in $\VE$ is uniquely determined by
  $\Tr\xi\in\TPoly$ and $Q_{T,k-4}\xi\in P_{k-4}(T)$.
\end{lemma}
\begin{proof} This is a direct consequence of \eqref{eq:CD}, Remark~\ref{rem:2DPolyData},
 Lemma~\ref{lem:BM} and the relation
    $Q_{\hat{T},k-4}(\xi\circ\Phi)=(Q_{T,k-4}\xi)\circ\Phi$.
\end{proof}
\begin{remark}\label{rem:2DInvariant}\rm
  Our definition of $\VE$, which is invariant under affine transformations,
   differs from the one in \cite{BM:2013:VEM}  for a general triangle.  The affine invariance simplifies
   the proofs of \eqref{eq:EhError} and \eqref{eq:EhPihError} in Section~\ref{subsec:2DEh}.
  We note that it is also possible to use the virtual finite element spaces from \cite{BM:2013:VEM}
   in the construction of $E_h$.
  But then the proofs of  \eqref{eq:EhError} and \eqref{eq:EhPihError} will become more involved.
\end{remark}
\begin{remark}\label{rem:2DVEM}\rm
  The definition of $H^2$ virtual element spaces on polygons
   and their applications to the plate bending problem can be found in  \cite{BM:2013:VEM,CM:2016:VEM}.
\end{remark}
%
\subsection{Construction on the Skeleton ${\G}={\bigcup}_{T\in\cT_h}{\p T}$}
\label{subsec:2DSkeleton}
 Given $v\in V_h$,
 our goal is to define $\fT$ representing (the desired) $E_hv\big|_{\p T}$ and $\gT$ representing
 (the desired) $(\p E_hv/\p n)\big|_{\p T}$ for all $T\in\cT_h$, such that
\begin{equation}\label{eq:2DTrace1}
  \text{$(\fT,\gT)\in\TPoly$ for all $T\in \cT_h$},
\end{equation}
 and the following conditions are satisfied:
\begin{align}
  &\text{if $T_1,T_2\in\cT_p$ (resp., $T_1,T_2\in\cT_e$), then
   $f_{v,\scriptscriptstyle T_1}(p)=f_{v,\scriptscriptstyle T_2}(p)$ (resp., $f_{v,\scriptscriptstyle T_1}=f_{v,\scriptscriptstyle T_2}$ on $e$),}
  \label{eq:2DTrace2}\\
  &\text{if two distinct  $T_1$ and $T_2$ belong to $\cT_e$,  then
   $g_{v,\scriptscriptstyle T_1}+g_{v,\scriptscriptstyle T_2}=0$ on $e$,}\label{eq:2DTrace3}\\
     &\text{if $e\in\cE_h^b$ is an edge of $T$ and $v\in H^1_0(\O)$, then $\fT=0$ on $e$.}
   \label{eq:2DTrace4}
\end{align}
 Note that \eqref{eq:2DTrace2} and \eqref{eq:2DTrace3} imply
 any piecewise $H^2$ function $\xi$ satisfying
 $(\xi_T,\p \xi_T/\p n)\big|_{\p T}=(\fT,\gT)$ for all $T\in\cT_h$ will belong to
 $H^2(\O)$, and \eqref{eq:2DTrace4} implies that $\xi\in H^2(\O)\cap H^1_0(\O)$ if $v\in H^1_0(\O)$.
\subsubsection{Construction at the Vertices}\label{subsubsec:2DVertices}
  In view of the compatibility conditions \eqref{eq:2DC2} and \eqref{eq:2DC3}, we need to
  define vectors $\bw_p\in\R^2$ associated with the vertices $p$ of $\cT_h$.
 There are three cases: (i) $p$ is an interior vertex,
 (ii) $p$ is boundary vertex that is not a corner of $\O$ and
 (iii) $p$ is a corner of $\O$.
\par\medskip\noindent
{\bf Case\sp(i)}
  For an interior vertex $p$, we define $\bw_p$ to be $\nabla \vT$, where
 $T$ is any triangle in $\cT_p$.
\par\medskip\noindent
{\bf Case\sp(ii)}
 For a boundary vertex $p$ that is not a corner of $\O$, we
 define $\bw_p$ to be $\nabla\vT$, where $T$ is one of the triangles in $\cT_p$ that
 has an edge on $\p\O$.  This choice ensures that $\bw_p\cdot\bm{t}=0$ if $v\in H^1_0(\O)$, where
 $\bm{t}$ is any vector  tangential to $\p\O$ at $p$.
\par\medskip\noindent
{\bf Case\sp(iii)}
 At a corner $p$ of $\O$, we define $\bw_p$  by
\begin{equation}\label{eq:GradCorner}
  \bw_p\cdot \bm{t}_i= (\p v/\p t_i)(p) \qquad\text{for $i=1,2$},
\end{equation}
 where $e_1, e_2\in\cE_h^b$ are the two edges emanating from $p$ and
 $\p/\p t_i$ is the derivative in the direction of the unit tangent $\bm{t}_i$ of $e_i$.
 Note that $\bw_p=0$
 at a corner $p$ of $\O$ if $v\in H^1_0(\O)$.
\par\medskip
 The choices of the triangles and tangent vectors in Case(i)--Case(iii) are illustrated in
 Figure~\ref{fig:2DSkeleton}.
\begin{figure}[h]
\includegraphics[height=1.5in]{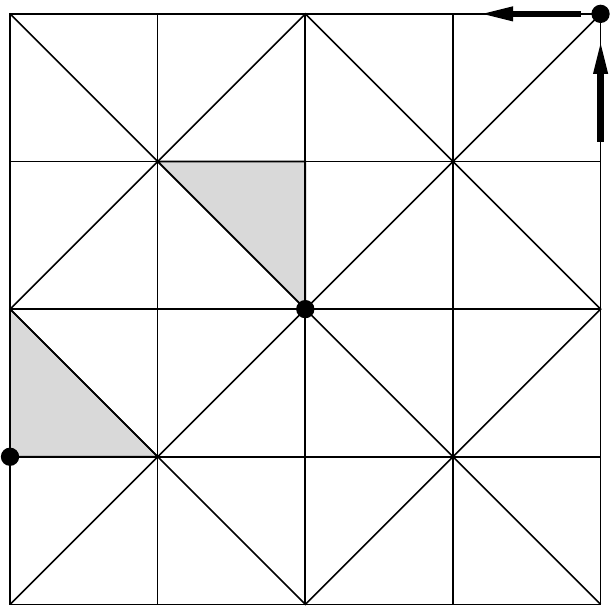}
\caption{Triangles and tangent vectors in the definition of $\bw_p$}
\label{fig:2DSkeleton}
\end{figure}
\begin{remark}\label{rem:C1Vertex}\rm
  If the condition
\begin{equation}\label{eq:C1Vertex}
   \nabla v_{\sss T_1}(p)=\nabla v_{\sss T_2}(p)\qquad\forall\,T_1,T_2\in\cT_p
\end{equation}
 is satisfied at a vertex $p$, then obviously $\bw_p=\nabla \vT(p)$ for all $T\in\cT_p$.
\end{remark}
\subsubsection{Construction on the Edges}\label{subsubsec:2DEdges}
\par
 On any edge $e\in\cE_h$, we define a polynomial $g_e\in P_{k-1}(e)$ as follows:
 First we choose $T\in\cT_e$ and then we specify that
\begin{align}
  &\text{$g_e(p)=\bw_p\cdot\bn_{e,\sss T}$ at an endpoint $p$ of $e$}.\label{eq:ge1}\\
  &\text{$g_e$ and $\p\vT/\p n$ have the same moments up to order $k-3$ on $e$.}\label{eq:ge2}
\end{align}
%
%
\subsubsection{Construction on the Triangles}\label{subsubsec:2DTriangles}
  We are now ready to define $(\fT,\gT)\in H^2(\p T)\times H^1(\p T)$ for any $T$ as follows.
 Given any edge $e$ of $T$, the function $\fT$ on $e$ is the unique polynomial in $P_{k}(e)$
 with the following properties:
\begin{align}
 &\text{$\fT$ agrees with $v$ at the two endpoints of $e$ and shares the same moments up to}
   \label{eq:2DSkeleton1}\\
 &\text{order $k-4$ with $v$,}\notag\\
 &\text{the directional derivative of $\fT$ at an endpoint $p$ of $e$ in the direction of the}
 \label{eq:2DSkeleton2}\\
   &\text{tangent $\bm{t}_e$ of $e$
   is given by $\bw_p\cdot\bm{t}_e$.}\notag
\end{align}
\begin{remark}\label{rem:vTEdge}\rm
  If the condition \eqref{eq:C1Vertex}
  is satisfied at both endpoints of $e$, then $\bw_{p}=\nabla\vT(p)$ at the two endpoints $p$ of
  $e$ by Remark~\ref{rem:C1Vertex} and then conditions \eqref{eq:2DSkeleton1}
  and \eqref{eq:2DSkeleton2} imply $\fT=v$ on $e$.
\end{remark}
\par
 Given any edge $e$ of $T$, we define
\begin{align}\label{eq:2DSkeleton3}
  &\text{$\gT=g_e$ if $T$ is the triangle chosen in the definition of $g_e$
  (cf. Section~\ref{subsubsec:2DEdges}),}\\
   &\text{otherwise $\gT=-g_e$.}\notag
\end{align}
\begin{remark}\label{rem:gTEdge}\rm
 If the condition \eqref{eq:C1Vertex} is satisfied at both endpoints of $e$ and $v$ is $C^1$
  across $e$, then Remark~\ref{rem:C1Vertex} and \eqref{eq:ge1}--\eqref{eq:ge2}
   imply that $\gT=\p\vT/\p n$ on
  $e$.
\end{remark}
\par
 By construction, the condition \eqref{eq:2DTrace1} is satisfied because
 the compatibility conditions \eqref{eq:2DC1}--\eqref{eq:2DC3} follow from \eqref{eq:ge1} and
 \eqref{eq:2DSkeleton1}--\eqref{eq:2DSkeleton2}.  The condition \eqref{eq:2DTrace2} follows from
 \eqref{eq:2DSkeleton1}--\eqref{eq:2DSkeleton2} and the condition \eqref{eq:2DTrace3} follows
 from \eqref{eq:2DSkeleton3}.  The choices we make in the definition of $\bw_p$ for $p\in\p\O$
 (cf. Case (ii) and Case (iii) in Section~\ref{subsubsec:2DVertices} and
 \eqref{eq:2DSkeleton1}--\eqref{eq:2DSkeleton2})
 also implies \eqref{eq:2DTrace4}.
%
\subsection{The Operator ${E}_{h}$}\label{subsec:2DEh}
 Let $v\in V_h$ and $T\in\cT_h$ be arbitrary, and $(\fT,\gT)\in \TPoly$  be
 the function pair constructed in
 Section~\ref{subsec:2DSkeleton}.  We define $E_h v\in\VE$ by the
 following conditions (cf. Lemma~\ref{lem:GeneralBM}):
\begin{equation}\label{eq:Eh}
  (E_hv,\p E_h v/\p n)=(\fT,\gT) \quad\text{on $\p T$} \quad
  \text{and} \quad Q_{T,k-4}(E_hv)=Q_{T,k-4}(v).
\end{equation}
\par
 It follows from \eqref{eq:2DTrace2}--\eqref{eq:2DTrace3} that the piecewise $H^2$ function
 $E_hv$ belongs to $H^2(\O)$, and \eqref{eq:2DTrace4} implies $E_hv\in H^1_0(\O)$ if
 $v\in H^1_0(\O)$. It only remains to establish the estimates
 \eqref{eq:EhError} and \eqref{eq:EhPihError}.
\par
 Note that
 Remark~\ref{rem:vTEdge} and Remark~\ref{rem:gTEdge} imply
\begin{align}\label{eq:C1Invariance}
  &\text{$(\fT,\gT)=(\vT,\p \vT/\p n)$ on $\p T$ if $v$ is $C^1$ on $\p T$,}
\end{align}
 and hence $v=E_hv$ if $v$ is $C^1$ on $\p T$, which is the rationale
 behind  \eqref{eq:EhError} and \eqref{eq:EhPihError}.
\begin{theorem}\label{thm:2DEhError}
 The estimate \eqref{eq:EhError} holds with a positive constant $C_\sharp$ that only depends on $k$ and the
 shape regularity of $\cT_h$.
\end{theorem}
\begin{proof} All the constants (explicit or hidden) that appear
 below will only depend on the minimum angle of $\cT_h$.
\par
 Let $T\in\cT_h$ be arbitrary.
 In view of Remark~\ref{rem:2DTPolyDimension}, Lemma~\ref{lem:GeneralBM} and the equivalence of
  norms on finite dimensional vector spaces, we have, by scaling,
\begin{align}\label{eq:LTwoNorm}
   \|\xi\|_{L_2(T)}^2&\approx \|Q_{T,k-4}\xi\|_{L_2(T)}^2+
    \sum_{e\in\cE_T}\big[\hT\|Q_{e,k-4}\xi\|_{L_2(e)}^2+\hT^3\|Q_{e,k-3}(\p\xi/\p n)\|_{L_2(e)}^2\big]\\
       &\hspace{40pt}+\sum_{p\in\cV_T}\big[\hT^2 \xi^2(p)+ \hT^4|\nabla\xi(p)|^2\big]
       \hspace{50pt}\forall\,\xi\in\VE,\notag
\end{align}
 where $\hT$ is the diameter of $T$ and $\cV_T$ (resp., $\cE_T$) is the set of the three vertices (resp., edges) of $T$.
    Moreover the affine invariance of
 $\VE$ (cf. Definition~\ref{def:2DVEM}) together with \eqref{eq:PullBack2} and \eqref{eq:ArtifcialGradient}
  implies that the hidden constants in \eqref{eq:LTwoNorm} only depend on the
 shape regularity of $T$.
\par
  It follows from \eqref{eq:2DSkeleton1}, \eqref{eq:Eh}  and
 \eqref{eq:LTwoNorm} that
\begin{equation}\label{eq:2DLocalEhError1}
  \|v-E_hv\|_{L_2(T)}^2\approx \sum_{p\in\cV_T}\hT^4|\nabla(v-E_hv)(p)|^2
     +\sum_{e\in\cE_T}\hT^3\|Q_{e,k-3}\p(v-E_hv)/\p n\|_{L_2(e)}^2,
\end{equation}
 and we also have, by the construction of $\bw_p$ in
  Section~\ref{subsubsec:2DVertices}, \eqref{eq:ge2}, \eqref{eq:2DSkeleton2},
  and \eqref{eq:2DSkeleton3},
\begin{align*}
  |\nabla(v-E_hv)(p)|^2&\leq C_1\sum_{e\in\cE_p}h_e^{-1}\|\jump{\p v/\p n}\|_{L_2(e)}^2,\\
  \|Q_{e,k-3}\p(v-E_hv)/\p n\|_{L_2(e)}^2&\leq \|\jump{\p v/\p n}\|_{L_2(e)}^2,
\end{align*}
 where $\cE_p$  is the set of all the edges in $\cE_h$ that share $p$ as a common vertex, and hence
\begin{equation}\label{eq:2DEhLocalError2}
  \|v-E_hv\|_{L_2(T)}^2\leq C_2 \hT^3\sum_{p\in\cV_T}\sum_{e\in\cE_p}\|\jump{\p v/\p n}\|_{L_2(e)}^2.
\end{equation}
\par
 We then deduce from \eqref{eq:2DEhLocalError2} and scaling that
\begin{equation}\label{eq:2DEhLocalError3}
   \|D^2(v-E_hv)\|_{L_2(T)}^2\leq C_3 \hT^{-1}\sum_{p\in\cV_T}\sum_{e\in\cE_p}\|\jump{\p v/\p n}\|_{L_2(e)}^2.
\end{equation}
 Note that, because of the affine invariance of $\VE$, the scaling constants behind
 \eqref{eq:2DEhLocalError3} only depend on the shape regularity of $T$.
\par
 The estimate \eqref{eq:EhError} follows immediately from \eqref{eq:hNorm} and
 \eqref{eq:2DEhLocalError3}.
\end{proof}
\begin{theorem}\label{thm:2DEhPihError}
  The estimate \eqref{eq:EhPihError} holds with a positive constant $C_\flat$ that only
  depends on $k$ and the shape regularity of $\cT_h$.
\end{theorem}
\begin{proof}
  Let $T\in\cT_h$ be arbitrary and $S_T$ (the star of $T$) be the interior of the union of the closures of all
  the triangles in $\cT_h$ that share a common vertex with $T$.
   If $\zeta\in H^{k+1}(\O)$ belongs to $P_k(S_T)$, then $\Pi_h\zeta=\zeta$
  in $S_T$ and
  hence $\zeta-E_h\Pi_h\zeta=\Pi_h\zeta-E_h\Pi_h\zeta=0$ on $T$ by \eqref{eq:2DEhLocalError2}.
 The estimate \eqref{eq:EhPihError} can then be established through the Bramble-Hilbert lemma
 \cite{BH:1970:Lemma,DS:1980:BH}.
\end{proof}
\section{The Three Dimension Case}\label{sec:3D}
 The construction of $E_h$ in three dimensions follows the same strategy as in Section~\ref{sec:2D},
 and our treatment will be brief regarding the
 results and arguments that are (almost) identical with the two dimensional case.
\subsection{Trace Spaces for a Tetrahedron}\label{subsec:3DLocalSpaces}
 Let $T$ be a tetrahedron with vertices $p_1,p_2,p_3,p_4$, and $F_i$ be the face of $T$ opposite $p_i$.
 Let $\ell$ be a nonnegative number.  A function $u$ belongs to the piecewise Sobolev space
 $H^\ell(\p T)$ if and only if $u_i$, the restriction of $u$ to $F_i$,
  belongs to $H^\ell(F_i)$ for
 $1\leq i\leq 4$.
\par
 For a function  $\phi$  defined on a face $F$ of the tetrahedron $T$, the
  planar gradient $\nabla_{F_j}\phi$ is defined by
    $$\nabla_{F_j}\phi=\nabla\tilde\phi-(\nabla\tilde\phi\cdot\bn_{\sss F,\sss T})\bn_{\sss F,\sss T},$$
 where $\tilde\phi$ is any extension of $\phi$ to a neighborhood of $F$ in $\R^3$.
\par
   The operator $\Tr:H^2(T)\longrightarrow H^\frac32(\p T)\times H^\frac12(\p T)$ is again
  defined by \eqref{eq:Trace} in a piecewise sense.  We want to characterize the image of
 $\ET$ in $\TSpace$  under the operator $\Tr$, for which we will need more notations and definitions.
\par
  The common edge of $F_i$ and $F_j$ is denoted by $e_{ij}(=e_{ji})$ and
  $e_{ij}^\perp$ denotes the two dimensional subspace of $\R^3$ perpendicular to $e_{ij}$.
  The outward unit normal on $F_j$ is denoted by $\bn_j$, and
  we denote by $\bt_{j,i}$ the unit vector tangential to $F_j$, perpendicular to $e_{ij}$ and
  pointing outside $F_j$ (cf. Figure~\ref{fig:3DGeometry}).
%
\begin{figure}[hh]
\includegraphics[width=3in]{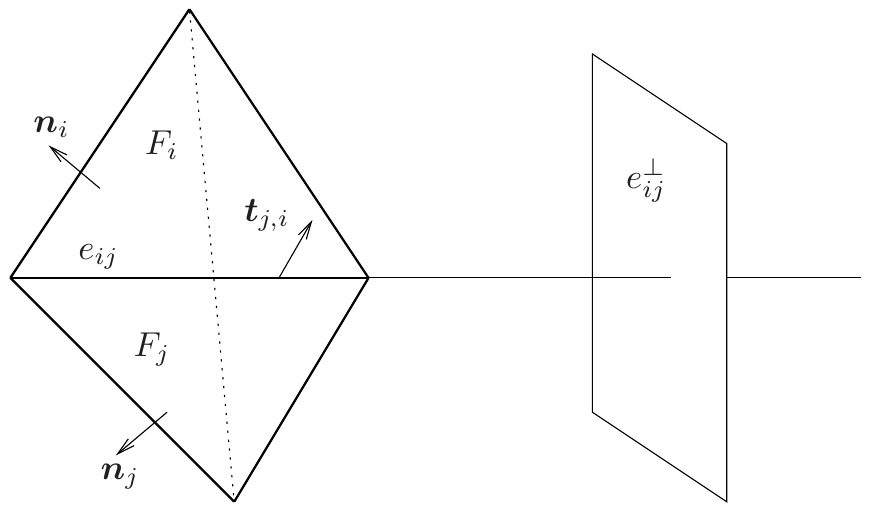}
\caption{Faces, normals, edge and orthogonal subspace}
\label{fig:3DGeometry}
\end{figure}
\begin{definition}\label{def:3DEdgeSpace}\rm
  The space $H^\frac12(e_{ij},e_{ij}^\perp)$ consists of all vector functions $\bw$ defined
  on $e_{ij}$ with image in $e_{ij}^\perp$ such that $\bw\cdot\bz\in H^\frac12(e_{ij})$ for all
   $\bz\in e_{ij}^\perp$.
\end{definition}
\begin{definition}\label{def:3DLocalTraceSpace}\rm
 A pair $(f,g)\in \HTO$ belongs to the space $\TSpace$ if and
 only if the following conditions are satisfied$\,:$
\begin{alignat}{3}
  f_i&=f_j&\qquad&\text{on $e_{ij}$ for $\;1\leq i \neq j\leq 4$},\label{eq:3DC1}\\
\intertext{and there exist $\bw_{ij}=\bw_{ji}\in H^\frac12(e_{ij},e_{ij}^\perp)$ such that}
  \nabla_{F_j} f_j\cdot \bt_{j,i}&=\bw_{ij}\cdot \bt_{j,i}&\qquad&\text{on $e_{ij}$
   for $\;1\leq i\neq j\leq 4$},
  \label{eq:3DC2}\\
  g_j&=\bw_{ij}\cdot\bn_j&\qquad&\text{on $e_{ij}$ for $\;1\leq i\neq j\leq 4$}.\label{eq:3DC3}
\end{alignat}
\end{definition}
%
\par
  Note that we can replace the compatibility conditions \eqref{eq:3DC2}--\eqref{eq:3DC3}
  by the condition
\begin{equation}\label{eq:3DCompatibleGradient}
   \nabla_{F_i}f_i+g_i\bn_i=\nabla_{F_j}f_j+g_j\bn_j\quad\text{on $\,e_{ij}\,$
   for $1\leq i\neq j\leq 4$.}
\end{equation}
\par
  It follows from the Sobolev Embedding Theorem that $\Tr\zeta\in\TSpace$ for
  $\zeta\in\ET$, where $\bw_{ij}$ is the orthogonal projection of
  $\nabla\zeta$ along $e_{ij}$ onto the subspace $e_{ij}^\perp$, and we can recover $\nabla\zeta$ on
  $F_i$ from $(f,g)=\Tr\zeta$ through the relation
\begin{equation}\label{eq:3DGradientRecovery}
  \nabla\zeta=\nabla_{F_i}f_i+g_i\bn_i\qquad\text{on $\,F_i\,$ for $1\leq i\leq 4$}.
\end{equation}
 We want to show that $\Tr\ET=\TSpace$.
\par
 Again we construct a linear bijection $\Phi^*:\TSpace\longrightarrow\tTSpace$ so that
 \eqref{eq:CD} is valid, where $\Phi$ is an orientation preserving affine transformation that maps
 the tetrahedron $\tT$ onto $T$.
 Let $(f,g)\in\TSpace$.  Motivated by \eqref{eq:CD}, \eqref{eq:ChainRule} and
 \eqref{eq:3DCompatibleGradient},  we define $\Phi^*(f,g)=(\tf,\tg)$, where
 $\tf$ is given by \eqref{eq:PullBack1}, $\tg$ is given by \eqref{eq:PullBack2} (where
 $\tilde\bn_i$ is the outward pointing unit normal along the face $\tilde F_i$) and
 the vector field $\mfg$ on $\p T$ is given by
\begin{equation}\label{eq:3DArtificialGradient}
  \mathfrak{g}=\nabla_{F_i}f_i+g_i\bn_i \qquad \text{on $\,F_i\,$ for $1\leq i\leq 4$}.
\end{equation}
\par
 It is straightforward to check that $(\tf,\tg)\in\tTSpace$, $\Phi^*:\TSpace\longrightarrow\tTSpace$ is a bijection,
 and that \eqref{eq:CD} follows from \eqref{eq:ChainRule}--\eqref{eq:PullBack2},
  \eqref{eq:3DGradientRecovery}
 and \eqref{eq:3DArtificialGradient}.
\par
 We can now establish the following analog of Lemma~\ref{lem:2DChracterization}.
\goodbreak
\begin{lemma}\label{lem:3DChracterization}
 The image of $\ET$ under $\Tr$ is the space $\TSpace$.
\end{lemma}
\begin{proof}
\par
 Given $(f,g)\in\HTO$ that satisfies \eqref{eq:3DC1}--\eqref{eq:3DC3},  we
 can reduce the construction of $\zeta$ to the following three cases by a partition of unity.
 (i)  $f$ and $g$ vanish near the vertices of $T$ and the edges of $T$, in which case we can use the
 operator $L_2$ in Lemma~\ref{lem:3DLifting}  to obtain $\zeta$.
 (ii) $f$ and $g$  are supported in a neighborhood of an edge and vanish near the vertices of $T$,
  in which case  we can assume through an affine transformation (cf. \eqref{eq:CD})  that the dihedral angle at
   the edge is a right angle and
   obtain $\zeta$ through Lemma~\ref{lem:2.5DLocalTrace}.
  (iii) $f$ and $g$ are supported near a vertex of $T$,  in which case we can assume through an affine transformation that
  the angle at the vertex is a solid right angle and obtain $\zeta$ through
   Lemma~\ref{lem:3DLocalTrace}.
\end{proof}
\subsection{Affine Invariant $\bm{H^2}$ Virtual Element Spaces}\label{subsubsec:3DVEK}
 We will use the same notation $\TSpace$ to denote $\Tr\ET$ for a tetrahedron $T$.  But the definition of
 $\TPoly$ is different.
\begin{definition}\label{def:3DTPoly}\rm Let $T$ be a tetrahedron.
 A pair $(f,g)\in\TSpace$  belongs to $\TPoly$ if and only if
$(f_i,g_i)\in \mathscr{V}^k(F_i)\times P_{k-1}(F_i)$
for $1\leq i\leq 4$.
\end{definition}
\begin{remark}\label{rem:3DTPolyDimension}\rm
 It follows from Lemma~\ref{lem:BM} and the constraints \eqref{eq:3DC1}--\eqref{eq:3DC3} that
 we need the following dofs for $\TPoly$: (i) The value of $v$ at each vertex $p$ together with the values of the
 three directional derivatives along  the three edges emanating from $p$, which requires  $4\times 4$ dofs.
 (ii) The moments of $v$ up to order $k-4$ on each edge, which together with (i)
 ensure the constraint \eqref{eq:3DC1}.  This requires $6\times (k-3)$ dofs.
     (iii) The moments of order up to $k-3$ on each edge
  in order to define, together with (i),
  a polynomial (vector) function of order $\leq k-1$
 on $e$ with images in $e^\perp$, which requires  $6\times 2(k-2)$ dofs.
  We can then use this polynomial (vector) function to define
  $\nabla_{F}\vF\cdot\bn_{e,F}$ on
   any edge $e$ of $F$ through \eqref{eq:3DC2}  and
 $\p v/\p n $ on $\p F$ through \eqref{eq:3DC3}.
 (iv) On each face $F$ we need
  to specify the moments of $v$ and $\p v/\p n$ up to order $k-4$
  in order to
 complete the definition of $\vF\in \mathscr{V}^{k}(F)$ and $\p v/\p n \in P_{k-1}(F)$,
 which requires $4\times 2\times \frac{(k-3)(k-2)}{2}$ dofs.
 Altogether we have
\begin{align}\label{eq:3DTPolyDimension}
\mathrm{dim}\,\TPoly&=16+6(k-3)+12(k-2)+4(k-3)(k-2)\\
&= 2(k-1)(2k+1).\notag
\end{align}
 The (visible) dofs of $\TPoly$ for $k=3$ and $4$ are depicted in Figure~\ref{fig:3DTSpace}, where (i)
 the values of $f$ at the vertices and
 the moments of $f$ on the edges and faces are represented by solid dots,
 and (ii)
 the directional derivatives of $f$ at the vertices and the moments of $g$ on the edges and faces are
 represented by arrows.
\end{remark}
\begin{figure}[h]
\includegraphics[width=4.5in]{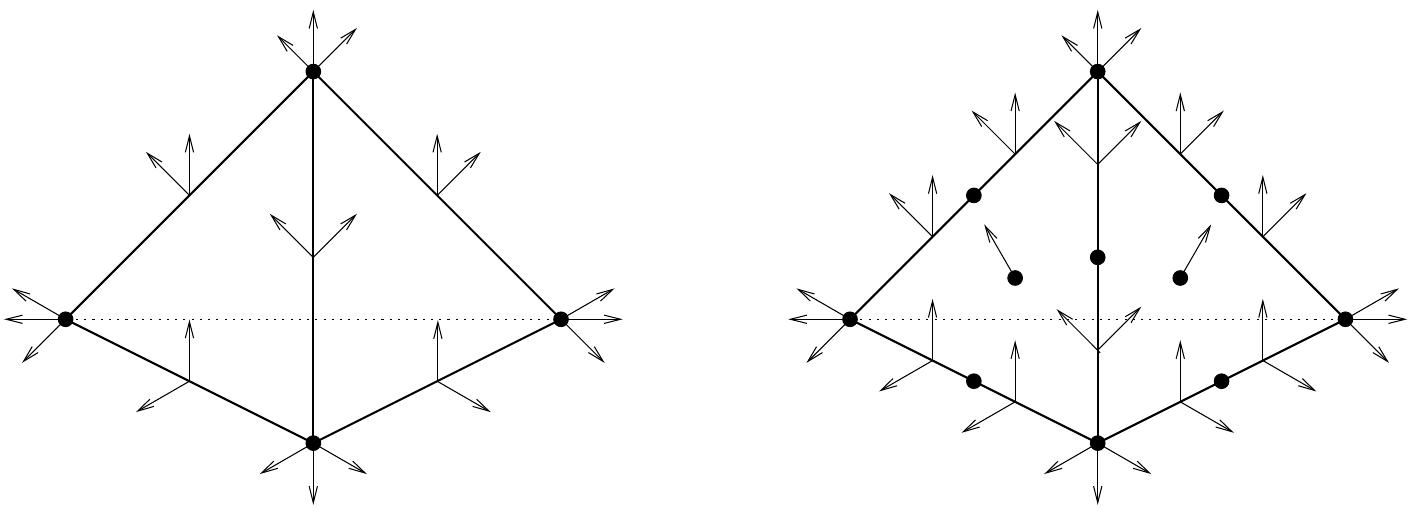}
\caption{Visible degrees of freedom for $(H^2\times H^1)_{3,2}(\p T)$ and $(H^2\times H^1)_{4,3}(\p T)$}
\label{fig:3DTSpace}
\end{figure}
%
%
 The well-posedness result in Lemma~\ref{lem:LocalWellPosedness} remains valid for a tetrahedron $T$
 and the definition of the virtual element space $\hatVE$ on the reference tetrahedron with vertices
 $(0,0,0)$, $(1,0,0)$, $(0,1,0)$ and $(0,0,1)$ is identical to the one in Definition~\ref{def:2DReferenceVEM}
 for the reference triangle.  The virtual element space $\VE$ for an arbitrary tetrahedron is then
 defined as in Definition~\ref{def:2DVEM} through an orientation preserving affine transformation $\Phi$ that
 maps $\hat{T}$ onto $T$, and Lemma~\ref{lem:BM} also holds for a tetrahedron.
\par
  The dimension of $\VE$ is now given by
\begin{align}\label{eq:3DVEMDimension}
 \mathrm{dim}\,\VE&=\dim\TPoly+\mathrm{dim}\,P_{k-4}(T)\\
 &=2(k-1)(2k+1)+\frac16(k-3)(k-2)(k-1)=\frac{(k-1)(k+1)(k+18)}{6}.\notag
\end{align}
\begin{remark}\label{rem:3DVE}\rm
  The definition of $\VE$ for a tetrahedron relies crucially on the fact that
   boundary data satisfying the
  compatibility condition \eqref{eq:3DC1}--\eqref{eq:3DC3} will belong to $\Tr H^2(T)$.
  Unlike the
  two dimensional case (cf. Remark~\ref{rem:ShortCut}),
   this cannot be taken for granted since macro elements of
  arbitrary order that share the same boundary data are yet to be developed.
\end{remark}
\begin{remark}\label{rem:3DVEM}\rm
 Three dimensional $H^2$ virtual elements on arbitrary  polyhedron have recently
 been proposed in \cite{BDR:2019:3DC1}.
\end{remark}
\subsection{Construction on the Skeleton ${\G}={\bigcup}_{T\in\cT_h}{\p T}$}
\label{subsec:3DSkeleton}
 Given any $v\in V_h$,
 we want to define $\fT$ representing (the desired) $E_hv\big|_{\p T}$ and $\gT$ representing
 (the desired) $(\p E_hv/\p n)\big|_{\p T}$ for all $T\in\cT_h$, such that
\begin{equation}\label{eq:3DTrace1}
  (\fT,\gT)\in\TPoly \qquad\forall\,T\in\cT_h
\end{equation}
 and the following conditions are satisfied:
\begin{align}
  &\text{if $T_1$ and $T_2$ belongs to $\cT_p$ (resp., $\cT_e$ or $\cT_F$),
  then $f_{v,\scriptscriptstyle T_1}(p)=f_{v,\scriptscriptstyle T_2}(p)$
  (resp.,  \hspace{30pt}\phantom{a}}
  \label{eq:3DTrace2}\\
  &\text{ $f_{v,\scriptscriptstyle T_1}=f_{v,\scriptscriptstyle T_2}$ on $e$
    or $f_{v,\scriptscriptstyle T_1}=f_{v,\scriptscriptstyle T_2}$ on $F$),}\notag\\
    &\text{if $T_1$ and $T_2$ are two distinct tetrahedra in $\cT_F$, then
   $g_{v,\scriptscriptstyle T_1}+g_{v,\scriptscriptstyle T_1}=0$ on $F$,}\label{eq:3DTrace3}\\
   &\text{if $F\in\cF_h^b$ is a face of $T$ and $v\in H^1_0(\O)$, then $\fT=0$ on $F$.}\label{eq:3DTrace4}
\end{align}
  Note that \eqref{eq:3DTrace2} and \eqref{eq:3DTrace3} imply
 any piecewise $H^2$ function $\xi$ satisfying
 $(\xi_T,\p \xi_T/\p n)\big|_{\p T}=(\fT,\gT)$ for all $T\in\cT_h$ will belong to
 $H^2(\O)$, and \eqref{eq:3DTrace4} implies that $\xi\in H^2(\O)\cap H^1_0(\O)$ if $v\in H^1_0(\O)$.
\subsubsection{Construction at the Vertices}\label{subsubsection:3DVertices}
 As in Section~\ref{subsec:2DSkeleton}, we first define the vectors $\bw_p$ associated with
 the vertices $p$ of $\cT_h$.   There are three cases: (i) $p$ is an interior vertex,
 (ii) $p$ is a boundary vertex that belongs to a face of $\O$,
 (iii) $p$ is a boundary vertex that does not belong to any face of $\O$.
\par\medskip\noindent
{\bf Case\sp(i)}
 For an interior vertex $p$, we choose a tetrahedron $T$ in $\cT_p$ and define $\bw_p$ to be $\nabla \vT$.
 \par\medskip\noindent
{\bf Case\sp(ii)}
 For a boundary vertex $p$ that belongs to a face $F$ of $\O$, we
 define $\bw_p$ to be $\nabla\vT$, where $T$ is a tetrahedron in $\cT_p$ that has a face on $F$.
 This choice ensures that $\bw_p\cdot\bm{t}=0$ if $v\in H^1_0(\O)$, where
 $\bm{t}$ is any vector tangential to $\p\O$ at $p$.
\par\medskip\noindent
{\bf Case\sp(iii)}
 In this case $p$ is either a corner of $\O$ or $p$ belongs to an edge of $\O$.
 We define $\bw_p$ implicitly by
\begin{equation}\label{eq:3DCorner}
  \bw_p\cdot \bm{t}_i=\frac{\p v}{\p t_i}(p) \qquad\text{for}\quad i=1,2,3,
\end{equation}
 where $\p/\p t_1$, $\p/\p t_2$ and $\p/\p t_3$ are the tangential derivatives along three edges
 $e_1, e_2, e_3\in\cE_h^b$ emanating from $p$ that are not coplanar.
 This choice of $e_1,e_2,e_3$ implies $\bw_p={\bf 0}$ if $v\in H^1_0(\O)$.
\begin{remark}\label{rem:3DC1Vertex}\rm
 Note that Remark~\ref{rem:C1Vertex} is also valid here, i.e., $\bw_p=\nabla \vT(p)$ for all $T\in\cT_p$ if
 $v$ is $C^1$ at the vertex $p$.
\end{remark}
\subsubsection{Construction on the Edges}\label{subsubsection:3DEdges}
 In view of \eqref{eq:3DC2} and \eqref{eq:3DC3},
 we also need to define polynomial vector functions $\bw_{e}:e\longrightarrow e^\perp$ on the
 edges $e\in\cE_h$.  There are three cases: (i) $e$ is an interior edge of $\cT_h$, (ii)
 $e$ is a subset of a face of $\O$ and (iii) $e$ is a subset of an edge of $\O$.
\par\medskip
\noindent
{\bf Case\sp(i)}
 Let $e$ belong to $\cE_h^i$.  We choose $T\in\cT_e$,
 and then define $\bw_e$ by the following conditions:
\begin{align}
  &\text{at an endpoint $p$ of $e$, $\bw_e(p)$ is the projection of $\bw_p$ on $e^\perp$.}
  \label{eq:3DEdge2}\\
  &\text{$\bw_e$ and the projection of $\nabla\vT$ on $e^\perp$  have the same
  moments along $e$ up to \hspace{20pt}\phantom{a}}
  \label{eq:3DEdge3}\\
  &\text{order $k-3$.}\notag
\end{align}
\par\medskip\noindent
{\bf Case\sp(ii)}
 Let $e$ be an edge of $\cT_h$ that is a subset of a face $F$ of $\Omega$.  We define $\bw_e$
 again by \eqref{eq:3DEdge2}--\eqref{eq:3DEdge3}, but with the stipulation that one of the
 faces of $T$ is a subset of $F$.
 This additional condition (together with the choices made in Cases (ii) and (iii) in
 Section~\ref{subsubsection:3DVertices})
 implies that
  $\bw_e\cdot\bm{t}=0$ on $e$ if $v\in H^1_0(\O)$, where  $\bm{t}$ is any vector
   tangential to $F$.
\par\medskip\noindent
{\bf Case\sp(iii)}
 Let $e$ be an edge of $\cT_h$ that is a subset of an edge of $\O$.  Then there are two distinct
 faces $F_1,F_2\in \cF_h^b\cap \cF_e$ and
 we  define $\bw_e$ by \eqref{eq:3DEdge2} together with the condition that
\begin{equation}\label{eq:3DEdge4}
  \text{$\bw_e\cdot\bm{n}_{e,\sss F_j}$ and
  $\nabla_{F_j}v_{\sss F_j}\cdot\bm{n}_{e,\sss F_j}$ have identical moments up to order
  $k-3$ for $j=1,2$.}
\end{equation}
%
  Our choices of $F_1$ and $F_2$
  (together with the choices made in Cases (ii) and (iii) in
 Section~\ref{subsubsection:3DVertices}) ensures that $\bw_e=0$ on $e$ if $v\in H^1_0(\O)$.
\begin{remark}\label{rem:3DC1Edge}\rm
 In the case where $v\in V_h$ is $C^1$ across an edge $e\in\cE_h$ and
 at the endpoints of $e$, it follows from Remark~\ref{rem:3DC1Vertex}
 and \eqref{eq:3DEdge2}--\eqref{eq:3DEdge4} that
 the vector field  $\bw_e$ is the projection of $\nabla \vT$ on $e^\perp$  for all $T \in\cT_e$.
\end{remark}
\subsubsection{Construction on the Faces}\label{subsubsec:Faces}
 We define $\gF$ on a face $F\in\cF_h$ as follows.  We choose $T\in\cT_F$ and
 stipulate that
\begin{align}
  &\text{on an edge $e$ of $F$, $\gF\in P_{k-1}(e)$ is given by $\bw_e\cdot\bn_{\sss F,\sss T}$},
  \label{eq:gF1}\\
  &\text{$\gF$ and $\p\vT/\p n$ have the same moments up to order $k-4$ on $F$.
  \hspace{50pt}\phantom{a}}\label{eq:gF2}
\end{align}
\begin{remark}\label{rem:3DgTEdge}\rm
 If  $v$ is $C^1$
  across $e\in\cE_h$ and at the endpoints of $e$, then we have $\gF=\p \vT/\p n$ on
  $e$ for all $F\in\cF_e$ and $T\in\cT_F$ by Remark~\ref{rem:3DC1Edge} and \eqref{eq:gF1}.
\end{remark}
%
\subsubsection{Construction on the Tetrahedra}
 We are now ready to define $(\fT,\gT)\in H^2(\p T)\times H^1(\p T)$ for any $T\in\cT_h$ as follows.
 On any edge $e$ of a face $F$ of $T$, $\fpF$ is the unique polynomial in $P_k(e)$ with the following properties:
\begin{align}
  &\text{$\fpF$ agrees with $v$ at the two endpoints of $e$ and share the same moments up}
  \label{eq:3DSkeleton1}\\
  &\text{to order $k-4$,}\notag\\
  &\text{the directional derivative of $\fpF$ at an endpoint $p$ of $e$ in the direction of the}
 \label{eq:3DSkeleton2}\\
   &\text{tangent $\bm{t}_e$ of $e$
   is given by $\bw_p\cdot\bm{t}_e$.}\notag
\end{align}
\begin{remark}\label{rem:3DvTEdge}\rm
  Remark~\ref{rem:vTEdge} is also valid here, i.e., $\fpF= v$ on $e$ if $v$ is $C^1$ at the endpoints of $e$.
\end{remark}
 Let $F$ be a face of $T$ and $e$ be an edge of $F$,  we define $\qF\in P_{k-1}(e)$ by
\begin{equation}\label{eq:3DSkeleton3}
  \qF=\bw_e\cdot\bn_{e,\scriptscriptstyle F}.
\end{equation}
%
%
 On each face $F$ of $T$,  the pair $(\fpF,\qF)$ belongs to $\FPoly$ by \eqref{eq:3DEdge2} and
 \eqref{eq:3DSkeleton1}--\eqref{eq:3DSkeleton3}.
 Hence we can define $\fF\in\mathscr{V}^k(F)$ to be the virtual element function
 (cf.  Lemma~\ref{lem:GeneralBM})
 that satisfies the following conditions:
\begin{equation}\label{eq:3DSkeleton4}
 \text{$\Tr \fF=(\fpF,\qF)$ on $\p F$ and $Q_{F,k-4}\fF=Q_{F,k-4}v$.}
\end{equation}
\begin{remark}\label{rem:uF}\rm
  If $v$ is $C^1$ on $\p F$, then $\qF=\p \vF/\p n$ on $\p F$ 
  by  Remark~\ref{rem:3DC1Edge} and \eqref{eq:3DSkeleton3}.
  It then follows from Remark~\ref{rem:PolynomialSubspace},
  Remark~\ref{rem:3DvTEdge} and \eqref{eq:3DSkeleton4} that
  $\fF=v$ on $F$.
\end{remark}
\par
 Given any face $F$ of $T$, we define
\begin{align}\label{eq:3DSkeleton5}
  &\text{$\gT=\gF$ if $T$ is the tetrahedron chosen in the definition of $\gF$
   (cf. }\hspace{30pt}\phantom{a}\\
   &\text{Section~\ref{subsubsec:Faces}), otherwise $\gT=-\gF$.}\notag
\end{align}
\begin{remark}\label{rem:3DgTFace}\rm
  If $v$ is $C^1$ on $\p T$, then \eqref{eq:gF2}, Remark~\ref{rem:3DgTEdge} and
  \eqref{eq:3DSkeleton5} imply
   $\gT=\p \vT/\p n$ on $\p T$.
\end{remark}
\par
 At the end of this process, we have constructed $(\fT,\gT)\in H^2(\p T)\times H^1(\p T)$
  for every polyhedron $T\in\cT_h$.  The pair $(\fT,\gT)$ belongs to
 $\TPoly$ because (i) the condition \eqref{eq:3DC1} is implied by
  \eqref{eq:3DSkeleton1}--\eqref{eq:3DSkeleton2}, (ii) the condition \eqref{eq:3DC2} is implied
  by \eqref{eq:3DSkeleton3}--\eqref{eq:3DSkeleton4}, and (iii) the  condition \eqref{eq:3DC3}
  is implied by \eqref{eq:gF1}.
 \par
  It follows from \eqref{eq:3DSkeleton1}--\eqref{eq:3DSkeleton4} that \eqref{eq:3DTrace2} is satisfied,
  and the condition \eqref{eq:3DTrace3} follows from
  \eqref{eq:3DSkeleton5}.  The choices we make in Section~\ref{subsubsection:3DVertices} and
  Section~\ref{subsubsection:3DEdges} ensure that $\fpF$ defined by \eqref{eq:3DSkeleton1}--\eqref{eq:3DSkeleton2}
 and $\qF$ defined by \eqref{eq:3DSkeleton3} both vanish
  on $\p F$ if the face
  $F$ of $T$ is a subset of $\p\O$ and $v\in H^1_0(\O)$.   The condition \eqref{eq:3DTrace4}  then
  follows from \eqref{eq:3DSkeleton4}.
\par
 In view of Remark~\ref{rem:uF} and Remark~\ref{rem:3DgTFace}
 the relation \eqref{eq:C1Invariance} remains valid, i.e., $E_hv=v$ if $v\in V_h$ is $C^1$ on $\p T$,
 which is the basis for the estimates \eqref{eq:EhError} and \eqref{eq:EhPihError}.
\subsection{The Operator ${E}_{h}$}\label{subsec:3DEh}
 We proceed as in Section~\ref{subsec:2DEh}.
  Let $v\in V_h$ and $T\in\cT_h$ be arbitrary, and $(\fT,\gT)\in \TPoly$  be
 the function pair constructed in
 Section~\ref{subsec:3DSkeleton}.  We define $E_h v\in\VE$  again by
 the conditions in \eqref{eq:Eh}, i.e.,
\begin{equation}\label{eq:3DEh}
  (E_hv,\p E_h v/\p n)=(\fT,\gT) \quad\text{on $\p T$} \quad
  \text{and} \quad Q_{T,k-4}(E_hv)=Q_{T,k-4}(v).
\end{equation}
\par
  It follows from \eqref{eq:3DTrace2}--\eqref{eq:3DTrace3} that
 $E_hv\in H^2(\O)$, and \eqref{eq:3DTrace4} implies that $E_hv\in H^1_0(\O)$ if $v\in H^1_0(\O)$.
\par
 The estimates  \eqref{eq:EhError} and \eqref{eq:EhPihError} are established by similar arguments as
 in Section~\ref{subsec:2DEh}, where the analog of \eqref{eq:LTwoNorm} for a tetrahedron $T$
 (cf. Lemma~\ref{lem:GeneralBM} and  Remark~\ref{rem:3DTPolyDimension}) is given by
\begin{align}\label{eq:3DL2Norm}
  \|\xi\|_{L_2(T)}^2&\approx \|Q_{T,k-4}\xi\|_{L_2(T)}^2
     +\sum_{F\in\cF_T}\hT\|Q_{F,k-4}\xi\|_{L_2(F)}^2
     +\sum_{F\in\cF_T}\hT^3\|Q_{F,k-4}(\p\xi/\p n)\|_{L_2(F)}^2\notag\\
     &\hspace{30pt}+\sum_{e\in\cE_T}\hT^2\|Q_{e,k-4}\xi\|_{L_2(e)}^2
        +\sum_{e\in\cE_T}\hT^4\|Q_{e,k-3}(\nabla\xi)_{e^\perp}\|_{L_2(e)}^2\\
   &\hspace{60pt}+\sum_{p\in\cV_T}\big[\hT^3\xi^2(p)+\hT^5|\nabla\xi(p)|^2\big]\notag
\end{align}
 for all $\xi\in\VE$, where $\cF_T$ (resp., $\cE_T$ and $\cV_T$)  is the set of the four faces
 (resp., six edges and four vertices) of $T$ and $(\nabla\xi)_{e^\perp}$ is the orthogonal projection of
 $\nabla\xi$ onto the subspace of $\R^3$ perpendicular to $e$.    The hidden constants in
 \eqref{eq:3DL2Norm} only depend on the shape regularity of $\cT_h$ because of the affine invariance of
 the virtual element spaces.
\par
 It follows from \eqref{eq:3DSkeleton1}, \eqref{eq:3DSkeleton4},
 \eqref{eq:3DEh} and \eqref{eq:3DL2Norm}
 that we have the following analog of \eqref{eq:2DLocalEhError1}:
\begin{align}\label{eq:3DLocalEhError}
 \|v-E_hv\|_{L_2(T)}^2&\approx  \sum_{p\in\cV_T}\hT^5|\nabla(v-E_hv)(p)|^2
  +\sum_{e\in\cE_T}\hT^4\|Q_{e,k-3}\big(\nabla(v-E_hv)\big)_{e^\perp}\|_{L_2(e)}^2\\
    &\hspace{40pt}+\sum_{F\in\cF_T}\hT^3\| Q_{F,k-4}\jump{\p v/\p n}\|_{L_2(F)}^2.\notag
\end{align}
 We can then establish the three-dimensional analogs
 of Theorem~\ref{thm:2DEhError} and Theorem~\ref{thm:2DEhPihError} as in
 Section~\ref{subsec:2DEh}.
\section{Concluding Remarks}\label{sec:Conclusions}
 Following the approach of this paper (and with more patience and persistence), one can construct
 enriching operators $E_h$ that maps
 the totally discontinuous $P_k$ finite element space into $H^2(\O)$,
 where \eqref{eq:EhError}
 and \eqref{eq:EhPihError} are valid for $J(w,v)$ given by
\begin{alignat*}{3}
  J(w,v)&=\sum_{e\in\cE_h^i}\Big[h_e^{-3}\int_e \jump{w}\jump{v} ds
    +h_e^{-1}\int_e \jump{\p w/\p n}\jump{\p v/\p n}ds\Big]
  &\qquad&\text{for $d=2$},\\
  J(w,v)&=\sum_{F\in\cF_h^i}\Big[h_F^{-3}\int_e \jump{w}\jump{v} dS+
    \hF^{-1}\int_F \jump{\p w/\p n}\jump{\p v/\p n}dS\Big]
  &\qquad&\text{for $d=3$}.
\end{alignat*}
\par
  One can also construct $E_h:V_h\cap H^1_0(\O)\longrightarrow H^2_0(\O)$ such that
   \eqref{eq:EhError} and \eqref{eq:EhPihError} are valid, provided the sum in \eqref{eq:2DJump}
   (resp., \eqref{eq:3DJump}) is taken over $\cE_h$ (resp., $\cF_h$).  This can also be carried out
   for the totally discontinuous $P_k$ finite element space.
\par
 Lemma~\ref{lem:3DChracterization} is
  also of independent interest, since inverse trace theorems for
 polyhedral domains in $\R^3$  do not appear to be readily available in the literature.
\appendix
 \section{Inverse Trace Theorems for $\R_+^2$ and $\R_+^3$}\label{append:ITT}
 We consider inverse trace theorems for $\R_+^2$ and $\R_+^3$ with data on the boundaries of
 these domains (cf. Figure~\ref{fig:Appendix}).  We will rely on the results in Lemma~\ref{lem:2DLifting}
 and Lemma~\ref{lem:3DLifting} that
  follow from the construction of inverse trace operators through the
  Fourier transform \cite{Necas:2012:Direct,Wloka:1987:PDE} and the Paley-Wiener theorem
      \cite{Lax:2002:FA}.
\begin{figure}[h]
\includegraphics[width=4in]{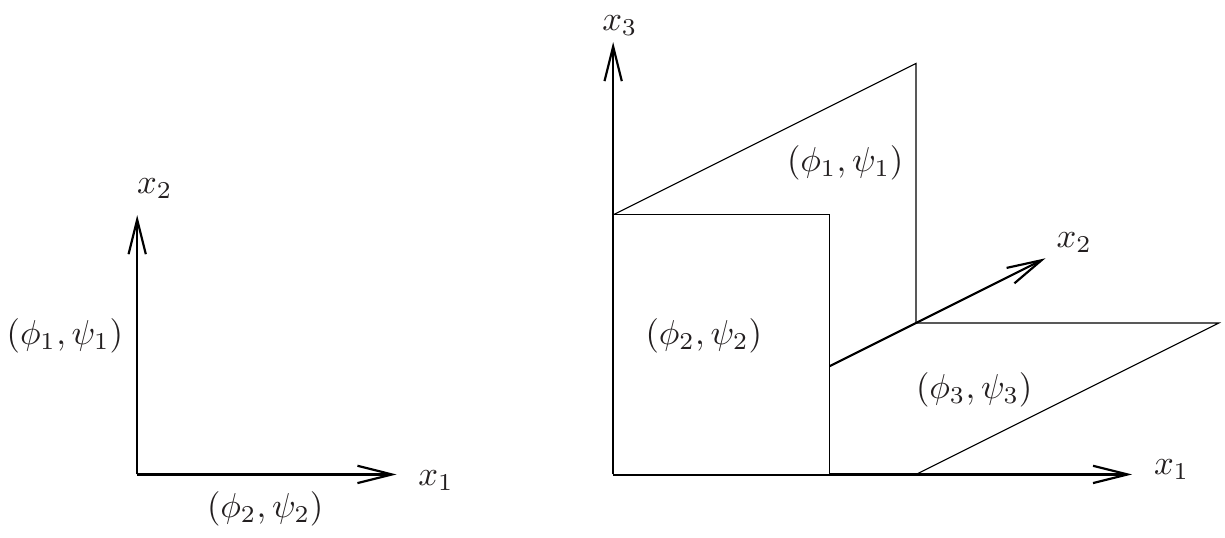}
\caption{Boundary data for $\R_+^2$ and $\R_+^3$}
\label{fig:Appendix}
\end{figure}
\begin{lemma}\label{lem:2DLifting}
  There exists a bounded linear operator $L_1:H^{2}(\R)\times H^{1}(\R)\longrightarrow
  H^{\frac52}(\R^2)$ such that $(i)$ $[L_1(\phi,\psi)](t,0)=\phi(t)$, $(ii)$ $[\p L_1(\phi,\psi)/\p x_2](t,0)=\psi(t)$,
  and $(iii)$  $L_1(\phi,\psi)(x_1,x_2)$ vanishes on the half plane $x_1<0$
  if  $\phi(t)$ and $\psi(t)$ vanish on the half line $t<0$.
\end{lemma}
\begin{lemma}\label{lem:3DLifting}
   There exists a bounded linear operator $L_2:H^{2}(\R^2)\times H^{1}(\R^2)\longrightarrow
  H^{\frac52}(\R^3)$ with the following properties: $(i)$ $[L_2(\phi,\psi)](x_1,x_2,0)=\phi(x_1,x_2)$,
  $(ii)$ $[\p L_2(\phi,\psi)/\p x_3](x_1,x_2,0)=\psi(x_1,x_2)$,
  and $(iii)$  $L_2(\phi,\psi)(x_1,x_2,x_3)$ vanishes on the half space $x_1<0$ $($resp., $x_2<0)$ if
  $\phi(x_1,x_2)$ and $\psi(x_1,x_2)$ vanish on the half plane $x_1<0$ $($resp., $x_2<0)$.
\end{lemma}
\par
 We begin with a two-dimensional inverse trace theorem.  We note that similar results for
 $H^2(\R_+^2)$ can be
 found in \cite[Section~1.5.2]{Grisvard:2011:EPN}.  Our approach is simpler
 (since we are considering  $H^\frac52(\R_+^2)$) and therefore its extension
 to three dimensions is easier.
\begin{lemma}\label{lem:2DLocalTrace}
  Let $(\phi_1,\psi_1)$ and $(\phi_2,\psi_2)$ belong to
  $H^{2}(\R_+)\times H^{1}(\R_+)$ such that
\begin{align}
  \phi_1(0)&=\phi_2(0),\label{eq:2DLocalTrace1}\\
    \psi_1(0)&=\phi_2'(0),\label{eq:2DLocalTrace2}\\
  \psi_2(0)&=\phi_1'(0). \label{eq:2DLocalTrace3}
\end{align}
  Then there exists $\zeta\in H^\frac52(\R_+\times\R_+)$ such that
\begin{equation}\label{eq:2DLocalTrace4}
   (\zeta,\p\zeta/\p x_i)=(\phi_i,\psi_i) \qquad \text{if $x_i=0$, $\,1\leq i\leq 2$}.
\end{equation}
\end{lemma}
\begin{proof}
\par
  First we extend $\phi_1$ and $\psi_1$ to $\R$, so that the extensions (still denoted by
  $\phi_1$ and $\psi_1$) satisfy $\phi_1\in H^{2}(\R)$ and  $\psi_1\in H^{1}(\R)$.
   This can be achieved by reflection (cf. \cite[Theorem~2.3.9]{Necas:2012:Direct} and \cite[Theorem~5.19]{ADAMS:2003:Sobolev}).
  Let $L_1$ be the lifting operator in Lemma~\ref{lem:2DLifting} and
  $\zeta_1=L_1(\phi_1,\psi_1)\in H^{\frac52}(\R^2)$ so that
 \begin{equation}\label{eq:2DZeta}
  \zeta_1(0,x_2)=\phi_1(x_2)\quad\text{and}\quad
  (\p\zeta_1/\p x_1)(0,x_2)=\psi_1(x_2).
 \end{equation}
   Then we define
  $\tilde \phi_2(x_1)=\phi_2(x_1)-\zeta_1(x_1,0)$, $\tilde \psi_2(x_1)=\psi_2(x_1)-(\p\zeta_1/\p x_2)(x_1,0)$
  for $x_1>0$.
\par
 Note that $\tilde \phi_2\in H^{2}(\R_+)$, and
\begin{equation*}
 \tilde\phi_2(0)=\phi_2(0)-\zeta_1(0,0)=\phi_2(0)-\phi_1(0)=0
\end{equation*}
 by   \eqref{eq:2DLocalTrace1} and \eqref{eq:2DZeta},
 and
\begin{equation*}
 \tilde\phi_2'(0)=\phi_2'(0)-(\p\zeta_1/\p x_1)(0,0)=\phi_2'(0)-\psi_1(0)=0
\end{equation*}
 by  \eqref{eq:2DLocalTrace2} and \eqref{eq:2DZeta}.
   Moreover we have $\tilde\psi_2\in H^{1}(\R_+)$, and
\begin{equation*}
 \tilde\psi_2(0)=\psi_2(0)-(\p\zeta_1/\p x_2)(0,0)=\psi_2(0)-\phi_1'(0)=0
\end{equation*}
 by \eqref{eq:2DLocalTrace3} and \eqref{eq:2DZeta}.
Hence their trivial extensions (still denoted by $\tilde\phi_2$ and $\tilde\psi_2$)    satisfy
 $\tilde\phi_2\in H^{2}(\R)$ and $\tilde\psi_2\in H^{1}(\R)$.
\par
 Let $\zeta_2=L_1(\tilde\phi_2,\tilde\psi_2)\in H^{\frac52}(\R^2)$ such that $\zeta_2(x_1,0)=\tilde\psi_2(x_1,0)$
 and $(\p \zeta_2/\p x_2)(x_1,0)=\tilde\psi_1(x_1)$.  Then $\zeta_2=0$ on the half plane $x_1<0$ by
 Lemma~\ref{lem:2DLocalTrace}, which implies
\begin{equation*}
  \zeta_2(0,x_2)=(\p\zeta_2/\p x_1)(0,x_2)=0 \qquad\forall\, x_2>0.
\end{equation*}
 We can now take $\zeta$ to be the restriction of
 $\zeta_1+\zeta_2$ to $\R_+\times\R_+$.
\end{proof}
\par
 Next we consider the three dimensional analog of Lemma~\ref{lem:2DLocalTrace}.
\begin{lemma}\label{lem:3DLocalTrace}
  Let $(\phi_1,\psi_1)$, $(\phi_2,\psi_2)$ and $(\phi_3,\psi_3)$ belong to
  $H^{2}(\RPTwo)\times H^{1}(\RPTwo)$
   such that the following conditions are satisfied$\,:$
\begin{alignat}{3}
  \phi_1(0,x_3)&=\phi_2(0,x_3)&\qquad&\forall\,x_3>0,\label{eq:phi1phi2}\\
  \phi_2(x_1,0)&=\phi_3(x_1,0)&\qquad&\forall\,x_1>0,\label{eq:phi2phi3}\\
  \phi_3(0,x_2)&=\phi_1(x_2,0)&\qquad&\forall x_2>0,\label{eq:phi3phi1}\\
  \psi_1(0,x_3)&=\frac{\p\phi_2}{\p x_1}(0,x_3)&\qquad&\forall\,x_3>0,\label{eq:psi1phi2}\\
  \psi_1(x_2,0)&=\frac{\p\phi_3}{\p x_1}(0,x_2)&\qquad&\forall\,x_2>0,\label{eq:psi1phi3}\\
  \psi_2(x_1,0)&=\frac{\p\phi_3 }{\p x_2}(x_1,0)&\qquad&\forall\,x_1>0,\label{eq:psi2phi3}\\
  \psi_2(0,x_3)&=\frac{\p\phi_1 }{\p x_2}(0,x_3)&\qquad&\forall\,x_3>0,\label{eq:psi2phi1}\\
  \psi_3(x_1,0)&=\frac{\p\phi_2}{\p x_3}(x_1,0)&\qquad&\forall\,x_1>0,\label{eq:psi3phi2}\\
  \psi_3(0,x_2)&=\frac{\p\phi_1}{\p x_3}(x_2,0)&\qquad&\forall\,x_2>0.\label{eq:psi3phi1}
\end{alignat}
 Then there exists $\zeta\in H^\frac52(\RPThree)$ such that
\begin{equation}\label{eq:3DLocalTrace1}
  (\zeta,\p\zeta/\p x_i)=(\phi_i,\psi_i) \qquad\text{if $\;x_i=0$}, \;1\leq i\leq 3.
\end{equation}
\end{lemma}
\begin{proof}
  First we extend $\phi_1$ and $\psi_1$ to $\R^2$ by reflection (twice)
  so that the extensions (still denoted by $\phi_1$ and $\psi_1$)
  satisfy $\phi_1\in H^2(\R^2)$ and $\psi_1\in H^1(\R^2)$.  Let $L_2$ be the lifting operator
  in Lemma~\ref{lem:3DLifting} and $\zeta_1=L_2(\phi_1,\psi_1)$ so that
\begin{align}
  \zeta_1(0,x_2,x_3)&=\phi_1(x_2,x_3),\label{eq:zeta11}\\
  (\p\zeta_1/\p x_1)(0,x_2,x_3)&=\psi_1(x_2,x_3).\label{eq:zeta12}
\end{align}
  Then we define, for $(x_1,x_3)\in \RPTwo$,
\begin{align}
 \tilde\phi_2(x_1,x_3)&=\phi_2(x_1,x_3)-\zeta_1(x_1,0,x_3),\label{eq:tildephi2}\\
  \tilde\psi_2(x_1,x_3)&=\psi_2(x_1,x_3)-(\p\zeta_1/\p x_2)(x_1,0,x_3). \label{eq:tildepsi2}
\end{align}
\par
 Note that $\tilde\phi_2$ belongs to
 $H^2(\RPTwo)$ and
\begin{equation*}
  \tilde\phi_2(0,x_3)=\phi_2(0,x_3)-\zeta_1(0,0,x_3)=\phi_2(0,x_3)-\phi_1(0,x_3)
  =0\qquad\text{for}\quad x_3>0
\end{equation*}
  by \eqref{eq:phi1phi2}, \eqref{eq:zeta11} and \eqref{eq:tildephi2}, and
\begin{equation*}
 \frac{\p \tilde\phi_2}{\p x_1}(0,x_3)=\frac{\p\phi_2}{\p x_1}(0,x_3)-\frac{\p\zeta_1}{\p x_1}(0,0,x_3) =
     \frac{\p\phi_2}{\p x_1}(0,x_3)-\psi_1(0,x_3)=0
\end{equation*}
  by \eqref{eq:psi1phi2}, \eqref{eq:zeta12} and \eqref{eq:tildephi2}.  Furthermore $\tilde\psi_2$ belongs to
 $H^1(\RPTwo)$ and
\begin{equation*}
 \tilde\psi_2(0,x_3)=\psi_2(0,x_3)- \frac{\p\zeta_1}{\p x_2}(0,0,x_3)
 =\psi_2(0,x_3)- \frac{\p\phi_1}{\p x_2}(0,x_3)=0\qquad\text{for}\quad x_3>0
\end{equation*}
  by \eqref{eq:psi2phi1}, \eqref{eq:zeta11} and \eqref{eq:tildepsi2}.
\par
 Hence we can extend $\tilde\phi_2$ and $\tilde\psi_2$ to $\R_+\times\R$ by reflection across $x_3=0$ (still denoted by
 $\tilde\phi_2$ and $\tilde\psi_2$) so that $\tilde\phi_2\in H^2(\R_+\times\R)$,
 $\tilde\psi_2\in H^2(\R_+\times\R)$,
 $\tilde\phi_2(0,x_3)=(\p \tilde\phi_2/\p x_1)(0,x_3)=0$ for $x_3\in\R$ and
 $\tilde\psi_2(0,x_3)=0$ for $x_3\in\R$.  Therefore
  the trivial extensions of $\tilde\phi_2$ and $\tilde\psi_2$ to
 $\R^2$ (still denoted by $\tilde\phi_2$ and $\tilde\psi_2$) belong to $H^2(\R^2)$ and $H^1(\R^2)$ respectively.
\par
 Let $\zeta_2=L_2(\tilde\phi_2,\tilde\psi_2)$.  Then we have, by Lemma~\ref{lem:3DLifting},
\begin{align}
 \zeta_2(x_1,0,x_3)&=\tilde\phi_2(x_1,x_3),\label{eq:zeta21}\\
 (\p\zeta_2/\p x_2)(x_1,0,x_3)&=\tilde\psi_2(x_1,x_3),\label{eq:zeta22}
\end{align}
  and
   $$\zeta_2(x_1,x_2,x_3)=0 \qquad\text{if $x_1<0$},$$
  which implies
\begin{equation}\label{eq:zeta2trace}
  \zeta_2=\p\zeta_2/\p x_1=0 \qquad \text{if $x_1=0$.}
\end{equation}
\par
 We  now define, for $(x_1,x_2)\in\R_+\times\R_+$,
\begin{align}
   \tilde\phi_3(x_1,x_2)&=\phi_3(x_1,x_2)-\zeta_1(x_1,x_2,0)-\zeta_2(x_1,x_2,0),\label{eq:tildephi3}\\
 \tilde\psi_3(x_1,x_2)&=\psi_3(x_1,x_2)-(\p\zeta_1/\p x_3)(x_1,x_2,0)-(\p\zeta_2/\p x_3)(x_1,x_2,0).
 \label{eq:tildepsi3}
\end{align}
 Then $\tilde\phi_3$  (resp., $\tilde\psi_3$) belongs to $H^2(\RPTwo)$
 (resp., $H^1(\RPTwo)$).
\par
 Moreover, it follows from \eqref{eq:phi3phi1},
  \eqref{eq:zeta11}, \eqref{eq:zeta2trace}  and \eqref{eq:tildephi3} that
\begin{equation*}
 \tilde\phi_3(0,x_2)=\phi_3(0,x_2)-\zeta_1(0,x_2,0)=
 \phi_3(0,x_2)-\phi_1(x_2,0)=0 \quad\text{for}\;x_2>0,
\end{equation*}
 and \eqref{eq:psi1phi3},  \eqref{eq:zeta12}, \eqref{eq:zeta2trace} and \eqref{eq:tildephi3}  imply
\begin{equation*}
\frac{\p\tilde\phi_3}{\p x_1}(0,x_2)=\frac{\p\phi_3}{\p x_1}(0,x_2)-\frac{\p\zeta_1}{\p x_1}(0,x_2,0)
=\frac{\p\phi_3}{\p x_1}(0,x_2)-\psi_1(x_2,0)=0\quad\text{for}\;x_2>0.
\end{equation*}
 From \eqref{eq:psi3phi1}, \eqref{eq:zeta11}, \eqref{eq:zeta2trace}
 and \eqref{eq:tildepsi3} we also have
\begin{equation*}
  \tilde\psi_3(0,x_2)=\psi_3(0,x_2)-\frac{\p\zeta_1}{\p x_3}(0,x_2,0)=
  \psi_3(0,x_2)-\frac{\p\phi_1}{\p x_3}(x_2,0)\quad\text{for}\;x_2>0.
\end{equation*}
\par
 Next we check the behavior of $\tilde\phi_3$ and $\tilde\psi_3$ at $x_2=0$.
 We have
\begin{equation*}
 \tilde\phi_3(x_1,0)=\phi_3(x_1,0)-\zeta_1(x_1,0,0)-\zeta_2(x_1,0,0)
 =\phi_3(x_1,0)-\phi_2(x_1,0)=0\quad\text{for}\;x_1>0
\end{equation*}
 by \eqref{eq:phi2phi3}, \eqref{eq:tildephi2}, \eqref{eq:zeta21} and \eqref{eq:tildephi3};
\begin{align*}
  \frac{\p\tilde\phi_3}{\p x_2}(x_1,0)&=\frac{\p\phi_3}{\p x_2}(x_1,0)
 -\frac{\p \zeta_1}{\p x_2}(x_1,0,0)-\frac{\p\zeta_2}{\p x_2}(x_1,0,0)\\
   &=\frac{\p\phi_3}{\p x_2}(x_1,0)-\psi_2(x_1,0)=0\quad\hspace{100pt}\text{for}\;x_1>0
\end{align*}
  by \eqref{eq:psi2phi3}, \eqref{eq:tildephi2}, \eqref{eq:zeta22} and \eqref{eq:tildephi3};
\begin{align*}
 \tilde\psi_3(x_1,0)&=\psi_3(x_1,0)-\frac{\p\zeta_1}{\p x_3}(x_1,0,0)-\frac{\p\zeta_2}{\p x_3}(x_1,0,0)\\
     &=\psi_3(x_1,0)-\frac{\p \phi_2}{\p x_3}(x_1,0)=0 \quad\hspace{100pt}\text{for}\;x_1>0
\end{align*}
 by \eqref{eq:psi3phi2}, \eqref{eq:tildephi2}, \eqref{eq:zeta21} and \eqref{eq:tildepsi3}.
\par
 The calculations above show that
  $\tilde\phi_3=\p \tilde\phi_3/\p n=\tilde\psi_3=0$ on
 the boundary of $\RPTwo$.  Hence
 their trivial extensions to $\R^2$
 (still denoted by $\tilde\phi_3$ and $\tilde\psi_3$) belongs to $H^2(\R^2)$ and $H^1(\R^2)$.
\par
 Let $\zeta_3=L_2(\tilde\phi_1,\tilde\psi_1)$.  Then we have, by Lemma~\ref{lem:3DLifting},
  $\zeta_3\in H^3(\R^3)$,
\begin{align}
  \zeta_3(x_1,x_2,0)&=\tilde\phi_3(x_1,x_2),\label{eq:zeta31}\\
  (\p\zeta_3/\p x_3)(x_1,x_2,0)&=\tilde\psi_3(x_1,x_2),\label{eq:zeta32}
\end{align}
 and
  $$\zeta_3(x_1,x_2,x_3)=0 \qquad\text{ if $x_1<0$ or $x_2<0$},$$
  which implies
\begin{equation}\label{eq:zeta3Trace}
 \zeta_3=\frac{\p\zeta_3}{\p x_1}=0 \quad\text{if $x_1=0$} \quad \text{and} \quad
 \zeta_3=\frac{\p\zeta_3}{\p x_2}=0\quad\text{if $x_2=0$}.
\end{equation}
\par
   We can now take $\zeta$ to be the restriction of $\zeta_1+\zeta_2+\zeta_3$ to $\RPThree$,
   and \eqref{eq:3DLocalTrace1} follows from
   \eqref{eq:zeta11}--\eqref{eq:zeta3Trace},
\end{proof}
\par
 Finally we have a three-dimensional result that is two-dimensional in nature and which can be derived by using
 the arguments in the proof of either Lemma~\ref{lem:2DLocalTrace} or Lemma~\ref{lem:3DLocalTrace}.
\begin{lemma}\label{lem:2.5DLocalTrace}
 Let $(\phi_1,\psi_1)$ and $(\phi_2,\psi_2)$ belong to $H^2(\R_+\times\R)\times H^1(\R_+\times\R)$
 such that
\begin{alignat*}{3}
    \phi_1(0,x_3)&=\phi_2(0,x_3)&\qquad&\forall\;x_3\in\R,\\
    \psi_1(0,x_3)&=\frac{\p \phi_2}{\p x_1}(0,x_3)&\qquad&\forall\;x_3\in\R,\\
    \psi_2(0,x_3)&=\frac{\p \phi_1}{\p x_2}(0,x_3)&\qquad&\forall\;x_3\in\R.
\end{alignat*}
 Then there exists $\zeta\in H^\frac52(\R_+\times\R_+\times\R)$ such that
\begin{equation*}
 (\zeta,\p\zeta/\p x_i)=(\phi_i,\psi_i)\qquad\text{if $x_i=0$, $1\leq i\leq 2$}.
\end{equation*}
\end{lemma}

\end{document}